\documentclass[11pt]{article}
 
\usepackage[margin=1in]{geometry} 
\usepackage{amsmath,amsthm,amssymb,todonotes}
\usepackage{accents, calc}
\usepackage{tikz}
\usepackage{caption}
\usepackage{subcaption}
\usetikzlibrary{quotes,arrows.meta}
\usepackage{pgfplots}
\usepackage{enumerate}
\usepackage{tabularx}
\usepackage[T1]{fontenc}
\usepackage{longtable}

\pgfplotsset{compat=1.14}
\usetikzlibrary{shapes}
\usetikzlibrary{quotes,arrows.meta}

\usepackage{pgfplots}
\usetikzlibrary{shapes}

\theoremstyle{plain}
\newtheorem{theorem}{Theorem}[section]
\newtheorem{lemma}[theorem]{Lemma}
\newtheorem{proposition}[theorem]{Proposition}
\newtheorem{corollary}[theorem]{Corollary}

\newtheorem{conjecture}[theorem]{Conjecture}
\theoremstyle{definition}
\newtheorem{definition}[theorem]{Definition}

\newtheorem{remark}[theorem]{Remark}

\newcommand{\eps}{\varepsilon}

\newcommand{\RR}{\mathbb{R}}
\newcommand{\FF}{\mathbb{F}}
\newcommand{\ZZ}{\mathbb{Z}}

\DeclareMathOperator{\dist}{dist}

\newcommand{\dbtilde}[1]{\accentset{\approx}{#1}}
 
\begin{document}

\title{New estimates on the size of $(\alpha,2\alpha)$-Furstenberg sets}
\date{}
\author{Daniel Di Benedetto \hspace{1.5cm} Joshua Zahl}

\maketitle

\begin{abstract}
We use recent advances on the discretized sum-product problem to obtain new bounds on the Hausdorff dimension of planar $(\alpha,2\alpha)$-Fursterberg sets. This provides a quantitative improvement to the $2\alpha+\epsilon$ bound of H\'era-Shmerkin-Yavicoli. In particular, we show that every $1/2$-Furstenberg set has dimension at least $1 + 1/4536$.  
\end{abstract}

\section{Introduction}
In \cite{w3}, Wolff defined a class of Besicovitch-type sets, inspired by the work of Furstenberg \cite{Fur}, which he called Furstenberg sets. 
\begin{definition}\label{defnFurstenberg}
For $0<\alpha\leq 1$, we say that a compact set $E\subset\RR^2$ is a (planar) \emph{$\alpha$-Furstenberg set} if for every direction $\omega\in S^1$, there is a line $\ell\subset\RR^2$ with direction $\omega$ and $\dim_H(E\cap \ell)\geq\alpha$. 
\end{definition}
Here and throughout, $\dim_H$ denotes Hausdorff dimension. In particular, every Besicovitch set in the plane is also a $1$-Furstenberg set. Wolff considered the problem of estimating 
\[
\gamma(\alpha) = \inf\{ \dim_H(E) : E \text{ is an } \alpha\text{-Furstenberg set}\},
\]
and showed that 
	\begin{equation}\label{wfurst}
	\max\Big\{2\alpha,\alpha+\frac{1}{2}\Big\}\leq \gamma(\alpha)  \leq \frac{3\alpha}{2} + \frac{1}{2}.
	\end{equation}
Wolff conjectured that the upper bound is sharp. When $\alpha=1/2$, the two lower bound estimates coincide, and we have $1\leq\gamma(1/2)\leq 5/4$. 

In the influential paper \cite{kt}, Katz and Tao connected the problem of estimating $\gamma(1/2)$ to the Falconer distance conjecture, and also to a discretized variant of the Erd\H{o}s-Szemer\'edi sum-product conjecture, which they called the discretized ring conjecture. This latter conjecture involves objects called $(\delta,\alpha)_n$-sets, which are discretized analogues of Borel subsets of $\RR^n$ that have Hausdorff dimension $\alpha$. To define these sets precisely, we introduce the parameters $n,\alpha, \eps, C, \delta$, which are chosen in that order. 
\begin{definition}
We say a set $E\subset\RR^n$ is a \emph{$(\delta,\alpha)_n$-set} if $E$ is a union of balls of radius $\delta$, and for each ball $B$ of radius $r\geq \delta$, we have $|E\cap B|\leq C\delta^{n-\eps}(r/\delta)^\alpha$.
\end{definition}
The discretized ring conjecture \cite[Conjecture 1.14]{kt} says that a $(\delta,1/2)_1$-set must expand substantially under addition and/or multiplication. 
\begin{conjecture}[Katz-Tao]\label{discRingConj}
There exists an absolute constant $c_1$, so that if $\eps>0$ is chosen sufficiently small and $C$ sufficiently large, the following holds for all sufficiently small $\delta>0$. Let $E\subset[1,2]$ be a $(\delta,1/2)_1$-set, with $|E|\geq \delta^{1/2+\eps}$. Then 
\[
|E+E| + |E.E| \geq \delta^{-c_1}|E|.
\]
\end{conjecture}
Katz and Tao \cite{kt} proved that a positive answer to Conjecture \ref{discRingConj} would imply that $\gamma(1/2)\geq 1+c_2$, for some (small) $c_2>0.$ Shortly thereafter, Bourgain \cite{bour2} established Conjecture \ref{discRingConj}.
\begin{theorem}[Bourgain]\label{bourgainThm}
Conjecture \ref{discRingConj} holds for some $c_1>0$.
\end{theorem}
As a consequence of \cite{kt} and \cite{bour2}, we have that $\gamma(1/2)\geq 1+c_2$. While the precise values of $c_1$ and $c_2$ were not computed explicitly in \cite{bour2,kt}, they are extremely small. Moreover, even if Conjecture \ref{discRingConj} was proved for the conjectured optimal value $c_1=1/2$, the reduction in \cite{kt} would still yield a bound for $\gamma(1/2)$ that is much smaller than the conjectured value $\gamma(1/2) = 5/4$.

In \cite{gkz}, Guth, Katz, and the second author found a simple new proof of Theorem \ref{bourgainThm} that also gives explicit quantitative estimates on the size of $c_1$.
\begin{theorem}[Guth-Katz-Zahl]\label{GKZTheorem}
Conjecture \ref{discRingConj} holds for all $c_1 < 1/136$.\footnote{One of the arguments in \cite{gkz} is inefficient, and Victor Lie observed that with a small modification the exponent can be improved to $1/128$.}
\end{theorem}
Armed with Theorem \ref{GKZTheorem} (or more precisely, a new, more technical variant), we prove new quantitative estimates on $\gamma(\alpha)$. First, we will introduce a class of Furstenberg-type sets, which have generated recent interest due to their connections with projection theory\footnote{Indeed, Theorem \ref{dimFurstenberg} below can be rephrased as a Kaufman-type projection theorem. See \cite{OrpSh} for related discussion.}. In the definition below, we identify a set $L$ of lines in $\RR^2$ with the corresponding points in the affine Grassmannian $A(2,1)$. Since the latter is a manifold, this allows us to define $\dim_H(L)$.

\begin{definition}[\cite{mr}][$(\alpha,\beta)$-Furstenberg set]
For $0<\alpha\leq 1,\ 0<\beta\leq 2$, we say that a set $E\subset\RR^2$ is a (planar) \emph{$(\alpha,\beta)$-Furstenberg set} if there exists a set $L$ of lines, with $\dim_H(L)\geq\beta$, so that $\dim_H(E\cap l)\geq\alpha$ for each $l\in L$. 
\end{definition}
We define
\[
\gamma(\alpha,\beta) = \inf\{ \dim_H(E) : E \text{ is an } (\alpha,\beta)\text{-Furstenberg set}\}.
\]
Numerous authors \cite{DOV, Her, HKM,shmerk, LuSt, mr, OrpSh} have obtained lower bounds for $\gamma(\alpha,\beta)$ in different regimes. The most relevant for our discussion is the following result of H\'era, Shmerkin, and Yavicoli, which used a variant of Theorem \ref{bourgainThm} due to Bourgain \cite{bour3} to obtain bounds on $\gamma(\alpha,2\alpha)$. 
\begin{theorem}[H\'era, Shmerkin, and Yavicoli]\label{HSYThm}
For each $0<\alpha< 1$, there exists $c(\alpha)>0$, which depends continuously on $\alpha$, so that $\gamma(\alpha,2\alpha)\geq 2\alpha+c(\alpha)$.
\end{theorem}
Note that when $\alpha=1/2$, this recovers the previous bound $\gamma(1/2)\geq 1+c_2$ (possibly for a different value of $c_2$) discussed above.

% Recently, Orponen and Shmerkin \cite{OrpSh} proved that $\tilde\gamma(\alpha,\beta)\geq 2\alpha+\eps_4(\alpha,\beta)$, where $\eps_4(\alpha,\beta)>0$ whenever $\beta>\alpha$, while D\k{a}browski, Orponen, and Villa \cite{DOV} proved that $\tilde\gamma(\alpha,2\alpha)\geq 2\alpha+(1-\alpha)(2\alpha-1)$; this is a quantitative improvement to Theorem \ref{HSYThm} whenever $\alpha>1/2$. 

While the authors in \cite{shmerk} do not explicitly compute $c(\alpha)$, the value obtained by their arguments would be extremely small. Using Theorem \ref{GKZTheorem} (or more precisely, a technical variant of the theorem), we obtain explicit estimates on the size of $c(\alpha)$. This estimate is stronger than what would be obtained by merely computing the value of $c(\alpha)$ from the arguments in \cite{shmerk}. 

\begin{theorem}\label{dimFurstenberg}
$\gamma(\alpha,2\alpha)\geq 2\alpha+\frac{\alpha(1-\alpha)}{6(155+68\alpha)}$. In particular, $\gamma(1/2)\geq \gamma(1/2,1)\geq 1 + 1/4536$.
\end{theorem}
Our bound is novel provided $\alpha$ is not much larger than $1/2$. For larger $\alpha$, the bound $\gamma(\alpha,2\alpha)\geq 2\alpha+(1-\alpha)(2\alpha-1)$ recently obtained by D\k{a}browski, Orponen, and Villa \cite{DOV} is superior.

\subsection{Thanks}
The authors would like to thank Victor Lie, Pablo Shmerkin, and Alexia Yavicoli for helpful comments and suggestions on a previous version of this manuscript. The authors would also like to thank the anonymous referee for many helpful comments, suggestions, and corrections. The first author was supported in part by a Four Year Doctoral Fellowship from the University of British Columbia. The second author was supported by a NSERC Discovery Grant.

\section{Discretized $(\alpha,\beta)$-Furstenberg sets}
To prove Theorem \ref{HSYThm}, H\'era, Shmerkin, and Yavicoli introduced an object they called a discretized $(\alpha,\beta)$-Furstenberg set. Before defining this object, we will introduce some additional notation. If $X\subset\RR^n$, we define the $\delta$-covering number $\mathcal{E}_{\delta}(X)$ to be the minimum number of balls of radius $\delta$ required to cover $X$. If $X$ is itself a union of balls of radius $\delta$, then $\mathcal{E}_{\delta}(X)\sim \delta^{-n}|X|$, where the implicit constant depends on $n$. Abusing notation, we say a set $X\subset\RR^n$ (which need not be a union of balls of radius $\delta$) is a $(\delta,\alpha)_n$-set if for each ball $B$ of radius $r\geq\delta$, we have $\mathcal{E}_{\delta}(E\cap B)\leq C\delta^{-\eps}(r/\delta)^\alpha$. In particular, if $X$ is a union of balls of radius $\delta$, then it is a $(\delta,\alpha)_n$-set in the sense given here if and only if it is a $(\delta,\alpha)_n$-set using the previous definition, though the associated constant $C$ will differ in these two definitions by a multiplicative factor depending on $n$. Since $n\in\{1,2\}$ throughout this paper, this distinction will not be important.

Next, if $L$ is a set of lines in $\RR^2$, we can identify $L$ with a subset of the affine Grassmannian $A(2,1)$. For concreteness, we will parameterize a coordinate patch of $A(2,1)$ as follows: if $l$ is the line $y=mx+b$ with $(m,b)\in [0,1]^2,$ we define $\iota(l)=(m,b)$ and $\operatorname{dir}(l) = (1,m)/\sqrt{1+m^2}$. In the discussions that follow, we can assume that every line is of the above form. We define the distance between two lines $l,l'$ to be the Euclidean distance between the points $\iota(l)$ and $\iota(l'),$ and we define $\angle(l,l')=\angle(\operatorname{dir}(l),\operatorname{dir}(l'))$.

We say a set $L$ of lines is $\delta$-separated if $\operatorname{dist}(\iota(l),\iota(l'))\geq\delta$ for all distinct $l,l'\in L$. A set $L$ of lines is said to be a $(\delta,\beta)_2$-set if $\iota(L)\subset[0,1]^2$ is a $(\delta,\beta)_2$-set. With these definitions, we recall the following definition from \cite{shmerk}. 
\begin{definition}\label{discretizedFurstDef}
For $0<\alpha\leq 1,$ $0<\beta\leq 2$, $\eps>0$, $C>0$, and $\delta>0$, we say that a set $E\subset\RR^2$ is a (planar) \emph{discretized $(\alpha,\beta)$-Furstenberg set} if $E\supset \bigcup_{l\in L}P_l$, where:
\begin{enumerate}[(i)]
\item\label{firstItemDiscretizedFurstenbergDefn} $L$ is a $\delta$-separated, $(\delta,\beta)_2$-set of lines (with implicit parameters $\eps$ and $C$), and $\# L\geq C^{-1}\delta^{\eps-\beta}$.
\item\label{secondItemDiscretizedFurstenbergDefn} For each $l\in L,$ $\mathcal{P}_l\subset l$ is a $\delta$-separated $(\delta,\alpha)_1$-set (with implicit parameters $\eps$ and $C$), and $\#\mathcal{P}_l \geq C^{-1}\delta^{\eps-\alpha}$.
\end{enumerate}
\end{definition}
The authors in \cite{shmerk} proved Theorem \ref{HSYThm} by showing that every discretized $(\alpha,2\alpha)$-Furstenberg set must have large $\delta$ covering number. The relationship between these two statements is as follows.
\begin{proposition}[\cite{shmerk}, Lemma 3.3]\label{discretizedToHausdorff}
Let $0<\alpha < 1,$ $0<\beta<2$, and $s\geq 0$. Suppose that there exists $\eps>0$ and $C>0$ so that for all sufficiently small $\delta>0$, every discretized $(\alpha,\beta)$-Furstenberg set $E$ satisfies $\mathcal{E}_\delta(E) \geq \delta^{-s}.$ Then every $(\alpha,\beta)$-Furstenberg set has Hausdorff dimension at least $s$.
\end{proposition}

Thus, in order to prove Theorem \ref{dimFurstenberg}, it suffices to establish the corresponding statement for discretized Furstenberg sets. Our result in this direction is the following. 
\begin{proposition}\label{discretizedFurstenbergProp}
Define $c(\alpha)=\frac{\alpha(1-\alpha)}{6(155+68\alpha)}$. Then for every $0<\alpha<1$ and every $c<c(\alpha)$, there exists $\eps>0$ and $C>0$ so that the following holds for all $\delta>0$ sufficiently small. Let $E$ be a discretized $(\alpha,2\alpha)$-Furstenberg set (with implicit parameters $\eps$ and $C$). Then $\mathcal{E}_\delta(E) \geq \delta^{-2\alpha -c}.$
\end{proposition}

At a qualitative level, the connection between sum-product estimates, point-line incidences, and Furstenberg sets is well known. In \cite{bkt}, Bourgain, Katz, and Tao proved a sum-product estimate over $\FF_p$, and used this to obtain a point-line incidence bound in $\FF_p^2$. The argument used by H\'era, Shmerkin, and Yavicoli~\cite{shmerk} is also based on these ideas, but uses Bourgain's discretized projection theorem \cite{bour3} instead of the discretized sum-product theorem. Our main contribution is that our arguments are carefully structured to obtain an effective quantitative relationship between the discretized sum-product theorem and the size of $(\alpha,2\alpha)$-Furstenberg sets, and our result gives an understanding of the quantitative limits of current approaches to the problem. This is in contrast to previous arguments that were mostly optimized for clarify of exposition. While previous proofs do not claim explicit exponents for the size of $(\alpha,2\alpha)$-Furstenberg sets near $\alpha=1/2$, an examination of these arguments shows that the resulting estimates would yield an analogue of Proposition \ref{discretizedFurstenbergProp} with a value of $c(\alpha)$ that was several orders of magnitude smaller. Further discussion about the efficiency of our methods can be found in Remark \ref{discussionOfLossesRemark}. 

One major difficulty we encounter when proving Proposition \ref{discretizedFurstenbergProp} is that while both Theorem \ref{GKZTheorem} and Proposition \ref{discretizedFurstenbergProp} involve sets that satisfy non-concentration conditions, not all paths between these two non-concentration conditions are equally efficient. To overcome this difficulty, we construct a sequence of point-line configurations that allow us to translate between the two notions of non-concentration. See Figure \ref{countpl} for an overview of this translation step. 

There are two steps in particular that weaken the quantitative dependence between the exponent in Theorem \ref{GKZTheorem} and the value of $c(\alpha)$ from Proposition \ref{discretizedFurstenbergProp}. The first of these steps is a series of reductions that transforms a (hypothetical) counter-example to Proposition \ref{discretizedFurstenbergProp} into an arrangement of points and lines, where the points are contained in a projective Cartesian product, and then subsequently into an arrangement where the points are contained in a rectilinear Cartesian product. Our estimates worsen as the transformation between these latter two structures becomes more singular, and our arguments are carefully structured to mitigate this. The second step is an application of the Balog-Szemer\'edi-Gowers theorem to transform an arrangement of points and lines with many point-line incidences into an arrangement of points and lines where the arrangement of points has small projection in three linearly independent directions. Again, our arguments are structured to minimize the losses coming from the application of the Balog-Szemer\'edi-Gowers theorem.

\section{Tools from combinatorics}\label{toolsFromAddComb}

%\subsection{Tools from additive combinatorics}\label{toolsFromAddComb}
We will begin by introducing several standard tools from additive combinatorics. See~\cite{tv} for an introduction to the topic. We will also state and prove a more technical version of Theorem \ref{GKZTheorem}.

We begin with some definitions.  If $A,B$ are finite subsets of an Abelian group, and $E\subset A\times B$, we define the partial sum and partial difference sets
\[
A \overset{E}{+} B = \{a+b: (a,b)\in E\},\quad A \overset{E}{-} B = \{a-b: (a,b)\in E\}.
\]
If $E = A\times B$, then we will abbreviate this as $A+B$ (resp. $A-B$); this is the usual Minkowski sum and Minkowski difference. %If instead $A,B$ are subsets of a field, we will define the partial product and quotient sets analogously. 

We recall the following version of the Balog-Szemer\'edi-Gowers theorem. The version stated here follows by making a small modification to Lemma 2.2 from \cite{bg}. The exact formulation given below can be found in~\cite[Lemma 8]{jones}. 

Here and throughout, we write $X\gtrsim Y$ or $Y\lesssim X$ if there is an absolute constant $C>0$ such that $X\geq CY$. If $X\gtrsim Y$ and $X\lesssim Y$, we write $X\sim Y$.

\begin{lemma}[Balog-Szemer\'edi-Gowers]\label{bsg}
Let $A,B$ be finite subsets of an Abelian group, and let $E\subset A\times B$. Then there is a set $A'\subset A$ such that 
\[
\#A' \gtrsim \frac{\#E}{\#B},\quad\textrm{and}\quad \#(A'-A') \lesssim \frac{(\#A)^4(\#B)^3(A \overset{E}{-} B)^4}{(\#E)^5}.
\]
% and 
% \[

% \]
\end{lemma}

Next, we will introduce some notation to describe the additive and multiplicative energy of discretized sets. Recall that if $A,B$ are subsets of an Abelian group (with group operation $+$), we define the additive energy 
\[
E(A,B) = \#\{(a,a',b,b')\in A^2\times B^2\colon a+b = a'+b'\}=\sum_{a\in A}\sum_{b\in B}\#\big( (a+B)\cap (b+A)\big).
\]
%We define the multiplicative energy $E_\times(A,B)$ analogously. 
The following result is a simple consequence of the Cauchy-Schwarz inequality; a proof can be found in \cite[Corollary 2.10]{tv}. 
\begin{lemma}\label{energyCauchySchwartz}
Let $A,B$ be finite subsets of an Abelian group. Then 
\[
E(A,B)\leq E(A,A)^{1/2}E(B,B)^{1/2}.
\]
\end{lemma}

We use the notation $N_{\delta}(X)$ to mean the $\delta$-neighborhood of $X$. That is, 
\[
N_{\delta}(X) = \{y\in\RR^n \colon \inf_{x\in X} |x-y| \leq \delta\},
\]
where the ambient space $\RR^n$ should be clear from context.

If $A\subset\RR$ and $\delta>0$, we define $A_{\delta^+} = N_{\delta}(A)\cap(\delta\ZZ)$ and $A_{\delta^\times} = N_{2\delta}(A)\cap2^{\delta\ZZ}$; note that with this definition, for each $x\in A$ there is a point $y\in A_{\delta^+}$ with $|x-y|\leq\delta$. Similarly, since $\frac{d}{dx} 2^x<4$ for $x\in [1,2]$, every point $x\in[1,2]$ has distance less than $2\delta$ from $2^{\delta\ZZ}$; thus if $A\subset[1,2]$, then for each $x\in A$ there is a point $y\in A_{\delta^\times}$ with $|x-y|\leq 2\delta.$ In practice, we will always have $A,B\subset[1,2]$.

$A_{\delta^+}$ (resp. $A_{\delta^\times}$) is a subset of the Abelian group $\delta\ZZ$ (resp. $2^{\delta\ZZ}$). We define
\[
E_+^\delta(A,B) = E(A_{\delta^+}, B_{\delta^+}),\quad E_\times^\delta(A,B) = E(A_{\delta^\times}, B_{\delta^\times}).
\]

% If $A,B\subset[T^{-1},T]$ for some $T>1$, then we define 
% \[
% E_\times^\delta(A,B) = E_+^\delta(\log A, \log B ).
% \]
% Here $\log A = \{\log a\colon a\in A\}$ and similarly for $\log B$; we have $\log A,\log B\subset [-\log T, \log T]$. In particular, Lemma \ref{energyCauchySchwartz} also applies with $E_+^\delta$ or $E_\times^\delta$ in place of $E_+$. 

We conclude this section by stating the precise variant of Theorem \ref{GKZTheorem} that we will use to prove Proposition \ref{discretizedFurstenbergWithNonConcentrationProp}.

\begin{proposition}\label{modifiedGKZ}
For each $0<\alpha<1$ and $\eps>0$, the there exists $\delta_0>0$ so that the following holds for all $0<\delta\leq\delta_0$. Let $K_1,K_2,K_3,K_4\geq 1$, and let $A\subset[1,2]$ be a $\delta$-separated set satisfying the following properties
\begin{enumerate}
\item\label{itm:1} $\#(A) \geq K_1^{-1}\delta^{-\alpha}$.
\item\label{itm:2} If $J$ is an interval of length at least $\delta$, then $\#(A\cap J) \leq K_2|J|^{\alpha}\delta^{-\alpha}$.
\item\label{itm:3} $\mathcal{E}_{\rho}(A-A) \leq K_3 \mathcal{E}_{\rho}(A)$ for all $\delta\leq\rho\leq 1$.
\item\label{itm:4} $E_{\times}^\delta(A,A)\geq K_4^{-1}(\mathcal{E}_{\delta}(A))^3$.
\end{enumerate}
Then
\begin{equation}\label{conclusionGKZ}
K_1^6 K_2^{6+\alpha} K_3^{2(9+4\alpha)} K_4^{4(2+\alpha)} \geq  \delta^{-\alpha(1-\alpha)+\eps}.
\end{equation}
\end{proposition}
The proof of Proposition \ref{modifiedGKZ} is nearly identical to the proof of Theorem 1.1 from~\cite{gkz}. We will briefly remark on the differences between the two theorems. First, in \cite{gkz}, the authors set $K_3=K_4,$ and they did not track the explicit dependence on $K_1$ and $K_2$. Note that when $\alpha=1/2$, $K_1\sim K_2\sim 1$ and $K_3=K_4=K$, then \eqref{conclusionGKZ} asserts that $K\gtrsim \delta^{-1/128+\eps}$; this recovers the result from \cite{gkz}, with the improvement observed by Victor Lie.

Second, the statement of Theorem 1.1 from~\cite{gkz} replaces Item \ref{itm:3} with the weaker requirement that $\mathcal{E}_{\delta}(A-A) \leq K_3 \mathcal{E}_{\delta}(A),$ and replaces Item \ref{itm:4} with the stronger requirement that $\mathcal{E}_{\delta}(A.A)\leq K_4\mathcal{E}_{\delta}(A)$. The authors in \cite{gkz} begin by finding a large subset $A'\subset A$ that satisfies Item \ref{itm:3} (in the form stated above), and then use their upper bound on $\mathcal{E}_{\delta}(A.A)$ (which also holds with $A'$ in place of $A$) to obtain Item \ref{itm:4}, with $A'$ in place of $A$. In our setting, we are not allowed to replace $A$ with a subset $A'\subset A$, since Item \ref{itm:4} might fail for $A'$. This is why we impose the stronger condition in Item \ref{itm:3} at the outset. we will sketch a proof of Proposition \ref{modifiedGKZ} in Appendix \ref{modifiedGKZAppendix}.

\subsection{Moran-regular trees}
In this section, $k$ and $T$ will be large integers; later, we will fix a small number $\eps$ and set $k=\lfloor 1/\eps\rfloor,$ and we will select $T$ to be an integer so that $2^{-k(T+1)}\leq\delta< 2^{-kT}.$ With this choice of $k$ and $T$, and once $\eps$ is a fixed constant, we will have $2^k\leq 2^{1/\eps}\sim 1$ and $2^T\leq\delta^{-2\eps}$. 

\begin{definition}
A $2^{T}$-adic interval is a subset of $[0,1)$ of the form $[s 2^{-jT}, (s+1)2^{-jT})$ for some $j\geq 0$ and $0\leq s\leq 2^{T}$.  A $2^{T}$-adic square is a Cartesian product of $2^{T}$-adic intervals. In what follows, we will simply use ``square'' to refer to a $2^{T}$-adic square. 
\end{definition}

Note that two distinct squares are either disjoint, or one contains the other. If $S$ is a square of side-length $2^{-jT}$ that contains a square $S'$ of side-length $2^{-(j+1)T}$, then we say $S'$ is a child of $S$. 

\begin{definition}
If $A\subset[0,1)^d$ and if $S$ is a square, we say $S$ meets $A$ if $S\cap A\neq\emptyset$. We say $A$ is discretized at scale $2^{-kT}$ if every square of side-length $2^{-kT}$ intersects at most one point from $A$.
\end{definition}

\begin{definition}
If  $A\subset[0,1)^d$ is discretized at scale $2^{-kT}$, and if $N_1,\ldots,N_k$ are integers, we say that $A$ is Moran-regular with branching $(N_0,\ldots,N_{k-1})$ if for each $j=0,\ldots,k-1$, every square $S$ of side-length $2^{-jT}$ that meets $A$ has exactly $N_j$ children $S'\subset S$ that meet $A$. 
\end{definition}

The next lemma says that every discretized set has a large subset that is Moran-regular. See \cite[Lemma 3.4]{kelsh} or \cite[Lemma 2.2]{shm} for a detailed proof.
\begin{lemma}\label{moranRegularLem}
Let $A\subset[0,1)^d$ be discretized at scale $2^{-kT}$. Then there exists a set $A'\subset A$ and integers $N_0,\ldots,N_{k-1}$ so that $\#A'\geq (2dT)^{-k}(\#A)$, and $A'$ is Moran-regular with branching $(N_0,\ldots,N_{k-1})$. Furthermore, we can ensure that each of the integers $N_0,\ldots,N_{k-1}$ are of the form $2^z$ for some non-negative integer $z$. 
\end{lemma}

We can also refine a collection of discretized sets, so that they are all Moran regular with the same branching.
 
\begin{corollary}\label{moranRegularCor}
Let $\mathcal{A}$ be a collection of subsets of $[0,1)^d$, each of which is discretized at scale $2^{-kT}$. Then there exist integers $N_0,\ldots,N_{k-1}$; a set $\mathcal{A}'\subset\mathcal{A}$; and for each $A\in \mathcal{A}'$, a set $A'\subset A$ so that
\[
\sum_{A\in\mathcal{A}'}\#A' \geq (2dT)^{-2k} \sum_{A\in\mathcal{A}}\#A,
\]
and for each $A\in\mathcal{A}'$, $A'$ is Moran-regular with branching $(N_0,\ldots,N_{k-1})$.
\end{corollary}
\begin{proof}
Apply Lemma \ref{moranRegularLem} to each $A\in\mathcal{A}$, and let $A'\subset A$ and $(N^A_0,\ldots, N^A_{k-1})$ be the output of that lemma. Since each $N^A_j$ is a power of two between $2^0$ and $2^T$, there are $(T+1)^k$ possible values for $(N^A_0,\ldots, N^A_{k-1})$. The result now follows from pigeonholing. 
\end{proof}

Finally, we observe that if $A\subset[0,1)^d$ is discretized at scale $2^{-kT}$ and Moran-regular with branching $(N_0,\ldots,N_{k-1})$, then for each square $S$ of side-length $2^{-jT}$, either $S\cap A=\emptyset,$ or $\#(S\cap A)=\prod_{i=j}^{k-1}N_i.$

%%%
%%%
%%%

\section{A discretized incidence theorem}
In this section, we will prove a discretized incidence theorem for collections of points and lines satisfying certain non-concentration assumptions. In Section \ref{reductionSection} we will show that the incidence theorem proved in this section implies Proposition \ref{discretizedFurstenbergProp}.

Throughout this section, we fix $0<\alpha<1$. We will let $\eps>0$ and $\delta>0$ be small parameters. If $A$ and $B$ are quantities that depend on $\delta$, we write $A\lessapprox B$ if there are constants $C=C(\alpha)$ and $\delta_0=\delta_0(\alpha,\eps)>0$ so that $A\leq \delta^{-C\eps}B$ for all $0<\delta\leq\delta_0$. If $A\lessapprox B$ and $B\lessapprox A$, we write $A\approx B$. 

Let $P\subset [0,1]^2$ be a set of $\delta$-separated points and let $L$ be a set of $\delta$-separated lines in $\RR^2$ (as before, all lines will be of the form $y=mx+b$, with $(m,b)\in [0,1]^2$). We say that a point $p$ is \emph{incident} (or $\delta$-incident, if we wish to emphasize the role of $\delta$) to a line $l$ if $p\in N_{\delta}(l).$ We write $I_{\delta}(P,L)$ to denote the set of pairs $(p,l)\in P\times L$ of incident points. If $I\subset P\times L$, define $I(p) = \{l\in L\colon (p,l)\in I\}$, and define $I(l) = \{p\in P\colon (p,l)\in I\}.$

Our task is to bound the number of incidences between $P$ and $L$, under the following non-concentration conditions. Let $V\geq 1,$ and let $I\subset I_{\delta}(P,L)$ be a set of incidences with the following properties.  
\begin{enumerate}[(A)]
\item\label{itemA}
For each line $l\in L$, 
\begin{equation}\label{eachLineContainsManyPoints}
\# I(l) \geq  \delta^{-\alpha+\eps}V.
\end{equation}

\item \label{itemB}
For each $l\in L$ and each ball $B$ of radius $r\geq\delta$,
\begin{equation}\label{nonConcentrationPtsOnLine}
\# (I(l)\cap B) \leq  \delta^{-\eps}(r/\delta)^\alpha V.
\end{equation}

\item \label{itemC}
For each ball $B$ of radius $r\geq\delta$, 
\begin{equation}\label{nonConcentrationlines}
\#(B\cap \iota(L))\leq  \delta^{-\eps} (r/\delta)^{2\alpha}.
\end{equation}

\item \label{itemD}
Each point $p\in P$ is incident to approximately an average number of lines. More precisely,
\begin{equation}\label{eachPointMeetsManyLines}
\delta^{\eps} \frac{(\#I)}{(\# P)} \leq \#I(p) \leq \delta^{-\eps} \frac{(\#I)}{(\#P)}.
\end{equation}

\item\label{robustTransversalityOfLinesThruPtItem} 
For each point $p\in P$, the lines incident to $p$ point in a set of directions that satisfies a ``robust transversality'' condition. Specifically, for each unit vector $v\in S^1$, we have
\begin{equation}\label{robustTransversalityOfLinesThruPt}
\# \{ l \in I(p)\colon  \angle(l, v) \leq \delta^{\eps} \} \leq \frac{1}{2} \#I(p) .
\end{equation}

\item\label{nonConcentrationLinesThroughAPointItem} 
For each point $p\in P$, the lines incident to $p$ point in a set of directions that forms a $(\delta,\alpha)_1$ set. Specifically, for each unit vector $v\in S^1$ and each $r\geq \delta$, we have
\begin{equation}\label{nonConcentrationLinesThroughAPoint}
\# \{ l \in I(p)\colon \angle(l, v) \leq r \} \leq \delta^{-\eps}(r/\delta)^\alpha.
\end{equation}
\end{enumerate}

\begin{proposition}\label{discretizedFurstenbergWithNonConcentrationProp}
Let $P\subset [0,1]^2$ be a set of points, let $L$ be a set of lines, let $V\geq 1$, and let $I\subset I_{\delta}(P,L)$. Suppose that $P,L,$ and $I$ satisfy Properties (\ref{itemA})-(\ref{nonConcentrationLinesThroughAPointItem}).  Then
\begin{equation}\label{lowerBdCardP}
\# P \gtrapprox \delta^{-2\alpha-c(\alpha)}V(\delta^{2\alpha}\#L),\qquad c(\alpha) = \frac{\alpha(1-\alpha)}{6(155+68\alpha)}.
\end{equation}
\end{proposition}

\begin{remark}
There are a number of (explicit and implicit) parameters used in this statement which may make it difficult to parse for readers unfamiliar with the $\lessapprox$ notation. The parameter $\alpha$ defines the `dimension' of the sets one is interested in, and this is the first parameter to be specified. The key point of the proposition is that there is a number $c(\alpha)$ --- the `dimension gain' --- relating the size of $P$ to $L$. The parameter $\eps$ can be arbitrarily small and quantifies the error which is acceptable in the hypothesis, such that the conclusion holds up to a related error. Given a value of $\eps$, there exists a number $\delta_0(\alpha,\eps)$ such that the statement holds for all $0<\delta\leq\delta_0$. Implicit in the $\lessapprox$ notation is a constant $C$, and the claim is that the inequality \label{lowerBdCardP} holds up to a multiplicative factor of $\leq\delta^{C\eps}$. Crucially, $C$ is a constant which may depend on $\alpha$ but does not depend on $\eps$.
\end{remark}

\begin{remark}\label{epsVsLessapprox}
Note that Proposition \ref{discretizedFurstenbergWithNonConcentrationProp} as stated above formally implies the (superficially) stronger statement, where the inequalities ``$A\leq \delta^{-\eps}B$'' in \eqref{eachLineContainsManyPoints}, \eqref{nonConcentrationPtsOnLine}, \eqref{nonConcentrationlines}, \eqref{eachPointMeetsManyLines}, and \eqref{nonConcentrationLinesThroughAPoint} are replaced by ``$A\lessapprox B$.''
\end{remark}

\begin{remark}\label{V1Remark}
The most interesting case is when $V=1$ and $\#L\approx\delta^{-2\alpha}$. However, in order to use Proposition \ref{discretizedFurstenbergWithNonConcentrationProp} to obtain bounds on the size of $(\alpha,2\alpha)$-Furstenberg sets, we will need to apply a two-ends reduction and rescaling to ensure that \eqref{robustTransversalityOfLinesThruPt} holds. Because of this, we must also consider other values of $V\geq 1$, and sets of lines with $\#L<\delta^{-2\alpha}$. However, if $(P,L,I)$ is an incidence arrangement that satisfies Properties (\ref{itemA})-(\ref{nonConcentrationLinesThroughAPointItem}) for some $V\geq 1$, then we can construct a new incidence arrangement $(P',L',I')$ as follows: $P'$ is obtained by randomly and independently selecting each point $p\in P$ with probability $\delta^{\eps}V^{-1}$; $I' = I\cap (P'\times L)$; and $L'=L$. Then with high probability, $(P',L',I')$ will satisfy Properties (\ref{itemA})-(\ref{nonConcentrationLinesThroughAPointItem}) with $V=1$, and the conclusion \eqref{lowerBdCardP} for $P'$ with $V=1$ implies  \eqref{lowerBdCardP} for $P$ (with the original value of $V$). Thus when proving Proposition \ref{discretizedFurstenbergWithNonConcentrationProp}, we may assume that $V=1$. 
\end{remark}

The remainder of this section will be devoted to the proof of Proposition \ref{discretizedFurstenbergWithNonConcentrationProp}; as noted in Remark \ref{V1Remark}, we may suppose that $V=1$.  In the arguments that follow, we will make statements of the following type: If $l\in L$, then 
\begin{equation}\label{exampleOfGtrapprox}
\#\{(p,q)\in I(l)^2\colon  \operatorname{dist}(p,q)\approx 1\}\gtrapprox\delta^{-2\alpha}.
\end{equation}
What this means is that if the constant $C_1$ is chosen appropriately (possibly depending on $\alpha$, but not depending on $\eps$ or $\delta$), then there exists a constant $C_2$ (again, possibly depending on $\alpha$, but not depending on $\eps$ or $\delta$) so that for all $\delta$ sufficiently small we have 
\begin{equation}\label{exampleOfGtrapproxC1C2}
\#\{(p,q)\in I(l)^2,\ \operatorname{dist}(p,q)\geq\delta^{C_1\eps}\}\geq\delta^{C_2\eps}\delta^{-2\alpha}.
\end{equation}
For this example, \eqref{exampleOfGtrapproxC1C2} (and hence \eqref{exampleOfGtrapprox}) is true with $C_1 = 3/\alpha$ and $C_2 = 3$. Indeed, by Property (\ref{itemA}), $\#I(l)\geq \delta^{\eps-\alpha}.$ Once $p\in I(l)$ has been chosen, there are $\geq \delta^{\eps-\alpha}$ elements $q\in I(l)$. By Property (\ref{itemB}) with $r = \delta^{C_1\eps}$ and with $C_1$ as above, at most $\delta^{-\eps}(r/\delta)^\alpha \leq \delta^{2\eps-\alpha}$ such $q$ satisfy $\operatorname{dist}(p,q)\leq \delta^{C_1\eps}$. Thus if $\delta>0$ is sufficiently small (in this case $\delta<2^{-1/\eps}$ will suffice), then there are at least $\frac{1}{2}\delta^{2\eps-2\alpha}\geq \delta^{C_2\eps}\delta^{-2\alpha}$ pairs $(p,q)\in I(l)^2$ with $\operatorname{dist}(p,q)\geq\delta^{C_1\eps}$. Arguments of this type will be used frequently in our proof without further comment. 

In the arguments that follow, we will construct a sequence of increasingly complicated configurations of points and lines. The main steps of this construction are illustrated in Figure \ref{countpl}. An important tool for constructing these configurations is H\"older's inequality. For example, suppose $I'\subset I$. Define $m(l) = \#I'(l)$, let $k\geq 2$, and let $k'$ be the conjugate exponent to $k$. Then by H\"older's inequality we have
\[
(\#I')^k = \Big( \sum_{l\in L}1\cdot m(l) \Big)^k \leq (\#L)^{k/k'}\Big(\sum_{l\in L} (m(l))^k\Big),
\]
and thus
\[
\#\{(p_1,\ldots,p_k,l)\in P^k\times L\colon (p_i,l)\in I'\ \textrm{for}\ i=1,\ldots,k\}\geq (\#I')^k(\#L)^{-k/k'}.
\]
Arguments of this type will be used in our proof without further comment.

\begin{figure}
     \centering
     \begin{subfigure}[b]{0.45\textwidth}
         \centering
        \begin{tikzpicture}

\filldraw [gray] (0,0) circle (3pt);
\node[text width=0.2cm] at (0,-0.4) {$p$};

\filldraw [gray] (0.3,3*0.3/1.2) circle (3pt);

\filldraw [gray] (0.5,3*0.5/1.2) circle (3pt);

\filldraw [gray] (0.8,3*0.8/1.2) circle (3pt);

\filldraw [gray] (1.2,3) circle (3pt);

\node[text width=0.2cm] at (1.2,3.4) {$q$};

\draw (0,0) -- (1.2,3);

\draw[densely dashed] (-1,0) -- (1,0);

\end{tikzpicture}
         \caption{Step 1}
         \label{fig:step1}
     \end{subfigure}
     \hfill
     \begin{subfigure}[b]{0.45\textwidth}
         \centering
        \begin{tikzpicture}

\filldraw [gray] (0,0) circle (3pt);
\node[text width=0.2cm] at (0,-0.4) {$p_1$};

\filldraw [gray] (0.3,3*0.3/1.2) circle (3pt);

\filldraw [gray] (0.5,3*0.5/1.2) circle (3pt);

\filldraw [gray] (0.8,3*0.8/1.2) circle (3pt);

\filldraw [gray] (1.2,3) circle (3pt);

\node[text width=0.2cm] at (1.2,3.4) {$q$};

\filldraw [gray] (1,0) circle (3pt);
\node[text width=0.2cm] at (1,-0.4) {$p_2$};

\filldraw [gray] (1.04,15*0.04) circle (3pt);

\filldraw [gray] (1.1,15*0.1) circle (3pt);

\filldraw [gray] (2.2,0) circle (3pt);

\node[text width=0.2cm] at (2.2,-0.4) {$p_3$};

\filldraw [gray] (2,3*0.2) circle (3pt);

\filldraw [gray] (1.65,3*0.55) circle (3pt);

\filldraw [gray] (3,0) circle (3pt);
\node[text width=0.2cm] at (3,-0.4) {$p_4$};

\filldraw [gray] (2.6,3*0.4/1.8) circle (3pt);

\filldraw [gray] (2.2,3*0.8/1.8) circle (3pt);

\filldraw [gray] (1.6,3*1.4/1.8) circle (3pt);

\draw (0,0) -- (1.2,3);
\draw (1,0) -- (1.2,3);
\draw (2.2,0) -- (1.2,3);
\draw (3,0) -- (1.2,3);

\draw[densely dashed] (-0.5,0) -- (3.5,0);

\end{tikzpicture}
         \caption{Step 2}
         \label{fig:step2}
     \end{subfigure}

     \centering
     \begin{subfigure}[b]{0.45\textwidth}
         \centering
         \begin{tikzpicture}

\filldraw [gray] (0,0) circle (3pt);
\node[text width=0.2cm] at (0,-0.4) {$p_1$};

\filldraw [gray] (1.2,3) circle (3pt);

\node[text width=0.2cm] at (1.2,3.4) {$q$};

\filldraw [gray] (1,0) circle (3pt);
\node[text width=0.2cm] at (1,-0.4) {$p_2$};

\filldraw [gray] (2.2,0) circle (3pt);

\filldraw [gray] (3,0) circle (3pt);
\node[text width=0.2cm] at (3,-0.4) {$p_4$};

\filldraw [gray] (1.65,3*0.55) circle (3pt);
\node[text width=0.2cm] at (1.4,1.25) {$p_5$};

\draw (0,0) -- (1.2,3);
\draw (1,0) -- (1.2,3);
\draw (2.2,0) -- (1.2,3);
\draw (3,0) -- (1.2,3);

\draw[densely dashed] (-0.5,0) -- (3.5,0);

\end{tikzpicture}
         \caption{Step 3}
         \label{fig:step3}
     \end{subfigure}
     \hfill
     \begin{subfigure}[b]{0.45\textwidth}
         \centering
        \begin{tikzpicture}

\filldraw [gray] (0,0) circle (3pt);
\node[text width=0.2cm] at (0.3,-0.3) {$p_5$};

\filldraw [gray] (1.5,2.5) circle (3pt);

\node[text width=0.2cm] at (1.8,2.2) {$q$};

\draw[densely dashed] (0,-0.5) -- (0,3);
\draw[densely dashed] (-0.5,0) -- (3,0);
\draw (0,0) -- (1.5,2.5);
\draw (-0.5,2.5) -- (1.5,2.5);
\draw (1.5,-0.5) -- (1.5,2.5);
\draw (-0.5,0.5) -- (1.5,2.5);

\end{tikzpicture}
         \caption{Step 4}
         \label{fig:step4}
     \end{subfigure}
        \caption{The main steps in translating point-line incidences to a sum-product statement}
        \label{countpl}
\end{figure}
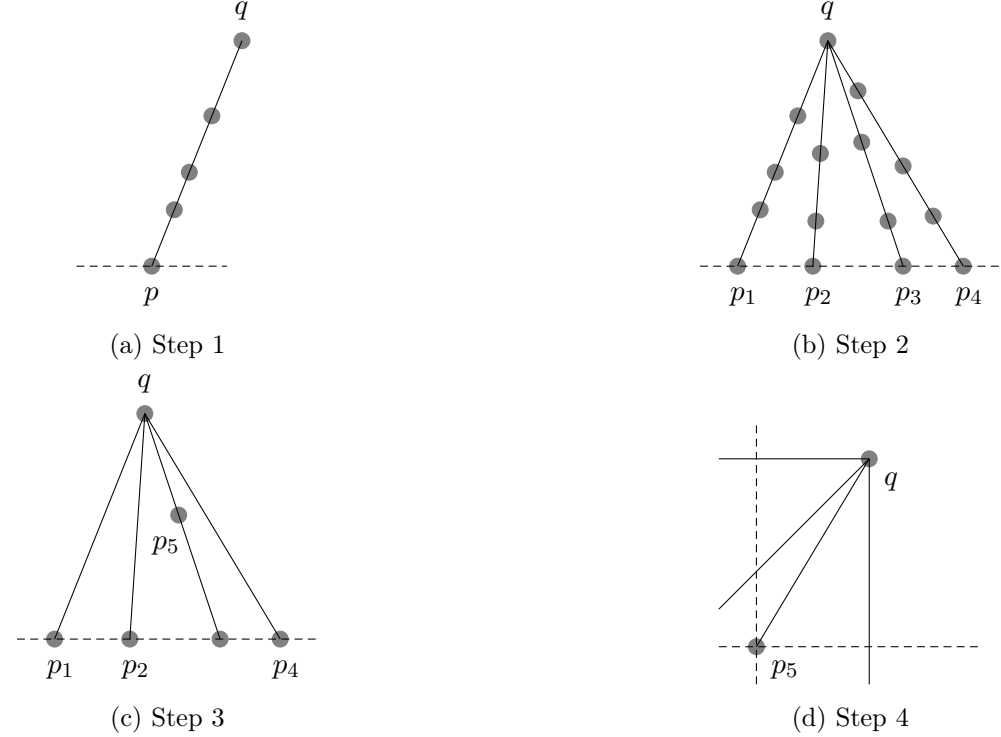

\begin{proof}[Proof of Proposition \ref{discretizedFurstenbergWithNonConcentrationProp}]
\item 
\paragraph{Counting triples relative to a fixed line.}
Let $l_0\in L$. We claim that 
\begin{equation}\label{triplesplq}
\begin{split}
\#& \{(p,l,q)\in I(l_0) \times L \times P : (p,l),(q,l) \in I; \operatorname{dist}(p, q)\gtrapprox1; \angle(l_0,l)\geq\delta^\eps;\\
& \textrm{there are}\ \gtrapprox \delta^{-\alpha}\ \textrm{points incident to}\ l\ \textrm{between}\ p\ \textrm{and}\ q\} \\
&\gtrapprox \Big(\delta^{-\alpha}\Big)\Big(\delta^{-\alpha}(\#L)(\#P)^{-1}\Big)\Big(\delta^{-\alpha}\Big)\\
&= \delta^{-3\alpha}(\#L)(\#P)^{-1}.
\end{split}
\end{equation}
The set in \eqref{triplesplq} is constructed as follows: first we select a point $p\in I(l_0)$; by Property (\ref{itemA}) there are $\gtrapprox \delta^{-\alpha}$ choices for $p$. Then we select a line $l\in I(p)$ with $\angle(l_0,l)\geq\delta^\eps$; by Properties (\ref{itemD}) and (\ref{robustTransversalityOfLinesThruPtItem}), there are $\gtrapprox \frac{(\#I)}{(\#P)}\approx \delta^{-\alpha}(\#L)(\#P)^{-1}$ choices for $l$. Finally we select a point $q\in I(l)$ with $\operatorname{dist}(p, q)\gtrapprox1$ and for which there are $\gtrapprox \delta^{-\alpha}\ \textrm{points incident to}\ l\ \textrm{between}\ p\ \textrm{and}\ q$. By Properties (\ref{itemA}) and (\ref{itemB}), there are $\gtrapprox \delta^{-\alpha}$ choices for $q$.

For each $q\in P$, let $L(q)\subset L$ be the set of lines $l$ so that there exists $p\in P$ such that $(p,l,q)$ belongs to the set defined in $\eqref{triplesplq}$. For each $l\in L(q)$, there are $\lessapprox 1$ points $p\in P$ so that $(p,l,q)$ belongs to the set defined in $\eqref{triplesplq}$ (indeed, such a point $p$ must be contained in $N_{\delta}(l)\cap N_{\delta}(l_0)$, and these two lines intersect at angle $\approx 1$).

\paragraph{Applying a two-ends reduction.}
Next we will apply a two-ends reduction; see \cite{TaoBlog} for an introduction to this topic. Fix $q\in P$, and for each $r\in[\delta,\pi]$ and $v\in S^1$, define 
\[
f(r,v)=r^{-\eps}\#\{l\in L(q)\colon \angle(l, v) \leq r\}.
\] 
This function is upper semi-continuous (indeed, for every $(r,v)\in [\delta,\pi]\times S^1$, there exists $\tau>0$ so that $f(r',v')\leq f(r,v)$ for all $|r-r'|<\tau,\ |v-v'|<\tau$), and thus achieves its maximum on the compact set $[\delta,\pi]\times S^1$. Let $(r_q,v_q)$ be a pair that achieves this maximum.

Define 
\[
L'(q) = \{l\in L(q)\colon \angle(l_i,v_q)\leq r_q\}.
\] 
Since $f(\pi, v)=\pi^{-\eps}\#L(q)$ for every $v\in S^1$, we have $f(r_q,v_q)\geq f(\pi, v_q) = \pi^{-\eps}\#L(q)$, and hence $\#L'(q)\geq (r_q/\pi)^{\eps}\#L(q)\geq (\delta/\pi)^{\eps}\#L(q)\gtrapprox\#L(q)$. 

We claim that for each $\delta\leq r\leq r_q$, there at most $(r/r_q)^\eps (\# L'(q))^4$ quadruples $(l_1,\ldots,l_4) \in (L'(q))^4$ so that $\angle(l_3, l_4)<r$. Indeed; once $l_1,l_2,l_3$ have been selected, $l_4$ must satisfy $\angle(l_3,l_4)\leq r$, and hence the number of such $l_4$ is at most $r^\eps f(r,\operatorname{dir}(l_3))$. Since $f(r,\operatorname{dir}(l_3))\leq f(r_q,v_q)$, we have $r^\eps f(r,\operatorname{dir}(l_3))\leq r^{\eps}f(r_q,v_q)= (r/r_q)^\eps \#L'(q)$. 

An identical argument for each pair of distinct indices $i\neq j$ shows there there are at most $\binom{4}{2} (r/r_q)^\eps (\# L'(q))^4$ quadruples $(l_1,\ldots,l_4) \in (L'(q))^4$ so that $\angle(l_i,l_j)<r$ for some $i\neq j$. This implies that if an estimate of the form
\begin{equation}\label{quadLineEq}
\#\{(x,l_1,\ldots,l_4)\in X\times (L'(q))^4\}\gtrapprox (\#X)(\#L(q))^4
\end{equation}
holds for the point $q$ (and some set $X$), then the same inequality (up to a constant multiplicative factor) must hold with the added condition that each distinct pair $l_i,l_j$ on the LHS of \eqref{quadLineEq} make angle $\gtrapprox r_q$.

\paragraph{Applying H\"older's inequality.}
Each point $q\in P$ has an associated number $r_q$ given by the two-ends reduction argument above. By applying dyadic pigeonholing to \eqref{triplesplq}, we can refine the set on the LHS of the inequality by a factor of $1/|\log\delta|$, and find a number $d_{l_0}$ such that $r_q\sim d_{l_0}$ for every point $q\in P$ that contributes to at least one triple from \eqref{triplesplq}. Note that inequality \eqref{triplesplq} remains true for this (slightly) smaller set of triples. Denote this refined set of triples by $T$.

Define 
\begin{equation*}
\begin{split}
\mathcal{N}_{l_0} = &\{(p_1,\ldots,p_4, l_1,\ldots,l_4, q)\in I(l_0)^4\times L^4\times P :\\
&   (p_i,l_i),(q,l_i)\in I;\ \dist(p_i,q) \approx 1;\ \angle(l_0,l_i)\geq\delta^{\eps},\ i=1,\ldots,4;\\
	& \qquad\operatorname{dist}(p_1,p_2), \operatorname{dist}(p_2,p_3), \operatorname{dist}(p_3,p_4) \leq d_{l_0} ;\\
	& \qquad\operatorname{dist}(p_1,p_2), \operatorname{dist}(p_2,p_3), \operatorname{dist}(p_3,p_4) \gtrapprox d_{l_0}\},
\end{split}
\end{equation*}
where in the above set we chose the labeling so that $p_1,\ldots,p_4$ appear in that order along $l_0$. We will show that \eqref{triplesplq} and the two-ends reduction argument described above implies a lower bound on the cardinality of this set.  We begin with the observation that
\[
\#T = \sum_{q\in P} \#\{(p,l)\in P\times L : (p,l,q)\in T\} \gtrapprox \delta^{-3\alpha}(\#L)(\#P)^{-1}.
\]
Applying H\"older's inequality to this sum gives
\begin{equation*}
\begin{split}
\#\{(p_1,\ldots,p_4, l_1,\ldots,l_4, q)\in I(l_0)^4\times L^4\times P :&(p_i,l_i,q) \in T, i=1,\ldots,4\} \gtrapprox \delta^{-12\alpha}(\#L)^4(\#P)^{-7}.
\end{split}
\end{equation*}
The quantity on the left-hand side of this inequality is at most $\lessapprox \delta^{-4\alpha}\#P$, so this estimate is of the form \eqref{quadLineEq}. Thus, thanks to our dyadic pigeonhole refinement and the two-ends reduction argument described above, we have
\begin{equation}\label{numberOf9Tuples}
\mathcal{N}_{l_0} \gtrapprox  \delta^{-12\alpha}(\#L)^4(\#P)^{-7}.
\end{equation}

Note that if $(p_1,\ldots,p_4, l_1,\ldots,l_4, q)\in \mathcal{N}_{l_0}$, then $\operatorname{dist}(l_0,q)\approx 1$. By \eqref{triplesplq}, we also have that there are $\gtrapprox \delta^{-\alpha}$ points $p_5\in P$ with $(p_5,l_3)\in I$ so that $p_5$ is between $p_3$ and $q$ on the line $l_3$. Thus by Property (\ref{itemB})  and the fact that $\angle(l_0,l_3)\approx 1$, we have that there are $\gtrapprox \delta^{-\alpha} $ points $p_5\in P$ so that (i) $(p_5,l_3)\in I$, (ii) $\operatorname{dist}(l_0,p_5)\approx 1$, (iii) $\operatorname{dist}(p_5, q)\approx 1$, and (iv) $p_5$ is contained inside the triangle $\Delta_{p_1,p_4,q}$ spanned by $p_1,p_4,q$. Items (i)--(iv) also imply that 
\begin{equation}
\begin{split}
\operatorname{dist}(q, \overline{p_1p_5})\gtrapprox d_{l_0},\quad\textrm{and}\quad
\operatorname{dist}(q, \overline{p_4p_5})\gtrapprox d_{l_0},
\end{split}
\end{equation} 
where $\overline{p_1p_5}$ denotes the line spanned by $p_1$ and $p_5$, and similarly for $\overline{p_4p_5}$. Thus, 
\begin{equation}\label{numberOf10Tuples}
\begin{split}
	&\#\{(p_1,\ldots,p_4,  l_1,\ldots,l_4, q,p_5)\in \mathcal{N}_{l_0}\times P : (p_5,l_3)\in I;\ \operatorname{dist}(l_0,p_5)\approx 1;\\
	&\qquad \operatorname{dist}(p_5,q)\approx 1;\ p_5\in \Delta_{p_1,p_4,q};\ \operatorname{dist}(p, \overline{p_1p_5})\gtrapprox d_{l_0};\ \operatorname{dist}(p, \overline{p_4p_5})\gtrapprox d_{l_0}\}\\
	& \gtrapprox \delta^{-13\alpha}(\#L)^4(\#P)^{-7}.
\end{split}
\end{equation}

\paragraph{Summing over the lines.}

Recall that the set described in \eqref{numberOf10Tuples} depends on the choice of line $l_0$. Let $\mathcal{D}$ be the disjoint union of such sets over all lines $l_0\in L$, i.e.
\begin{equation}\label{defnD}
\begin{split}
\mathcal{D} = \{&(l_0,p_1,\ldots,p_4, l_1,\ldots,l_4, q,p_5)\colon (p_1,\ldots,p_4, l_1,\ldots,l_4, q,p_5)\in \mathcal{N}_{l_0}\times P;\ (p_5,l_3)\in I;\\ 
&\operatorname{dist}(l_0,p_5)\approx 1;\ \operatorname{dist}(p_5,q)\approx 1;\  p_5\in \Delta_{p_1,p_4,q};\ 
\operatorname{dist}(p, \overline{p_1p_5})\gtrapprox d_{l_0};\ \operatorname{dist}(p, \overline{p_4p_5})\gtrapprox d_{l_0}\}.
\end{split}
\end{equation}
By \eqref{numberOf10Tuples}, we have 
\[
\#\mathcal{D}\gtrapprox \delta^{-13\alpha}(\#L)^5(\#P)^{-7}. 
\]

After dyadic pigeonholing, we can suppose there is a number $d$ so that so that $d_{l_0}\sim d$ for a $|\log\delta|^{-1}$ fraction of the tuples in $\mathcal{D}$. Abusing notation, we will continue to refer to this refined set as $\mathcal{D}$.

\paragraph{Fixing four special points.}

Next, we will use pigeonholing to select a choice of $(l_0,p_1,p_2,p_4,p_5)$ that occur in many tuples from $\mathcal{D}$. There are $\#L$ choices for $l_0$. Once $l_0$ has been specified, there are at most $\delta^{-\alpha-\eps}$ choices for $p_1$. Next, there are $\lessapprox d^\alpha\delta^{-\alpha-\eps}$ choices for each of $p_2$ and $p_4$. Finally, there are $\leq \#P$ choices for $p_5$. Thus there is a choice of $(l_0,p_1,p_2,p_4,p_5)$ that appears in 
\[
\gtrapprox \frac{ \delta^{-13\alpha}(\#L)^5(\#P)^{-7} }
{(\#L)(\delta^{-\alpha})(d^\alpha\delta^{-\alpha})^2(\#P) }= d^{-2\alpha}\delta^{-10\alpha}(\#L)^4(\#P)^{-8}
\]
tuples from $\#\mathcal{D}$. Fix this choice of $l_0,p_1,p_2,p_4,p_5$. Note that the point $p_3$ and line $l_3$ are determined up to multiplicity $\approx 1$. 

Write $p\sim q$ if there is a line $l\in L$ with $(p,l),(q,l)\in I$ (note that earlier we defined $X\sim Y$ if $X\lesssim Y$ and $Y\lesssim X$; it will always be clear from context which meaning of `$\sim$' is intended). Since $\operatorname{dist}(p_i,q)\approx 1$ for $i=1,2,4,5$, if $p_i\sim q$ then there are $\approx 1$ lines $l_i\in L$ incident to both $p_i$ and $q$. Thus if we define
\begin{equation*}
\begin{split}
Q_0 = \{ q \in P \colon p_1\sim q;\ p_2\sim & q;\ p_4\sim q;\ p_5\sim q;\ \operatorname{dist}(l_0, q)\approx 1;\ \operatorname{dist}(p_5,q)\approx 1;\\
& p_5\in\Delta_{p_1,p_4,q};\ \operatorname{dist}(q, \overline{p_1p_5})\gtrapprox d;\ \operatorname{dist}(q, \overline{p_4p_5})\gtrapprox d \},
\end{split}
\end{equation*}
then

\begin{equation}\label{goodPointsQ}
\#Q_0 \gtrapprox  d^{-2\alpha}\delta^{-10\alpha}(\#L)^4(\#P)^{-8}.
\end{equation}

Observe that each $q\in Q_0$ is incident a line, which in turn is incident to $p_1$ (i.e. each $q\in Q_0$ is in the `bush' of $p_1$). Thus by dyadic pigeonholing, we can select a set $Q\subset Q_0$ with $\#Q\gtrapprox \#Q_0$, so that each of the lines described above are incident to the same number of points from $Q$. This property will be needed to establish \eqref{consequenceOfEvenOnLine} below.

\paragraph{Rescaling.}
If $R$ is a rectangle, we define $\operatorname{dir}(R)\in S^1$ to be the direction of the major axis of $R$. In practice, we will always consider rectangles with eccentricity $>1$, so this is well defined.

Let $\mathcal{R}$ be a maximal set of rectangles of dimensions $4\times 100d$ that intersect $Q$; are centered at $p_1$; and point in $d$-separated directions. Since $|p_1-p_4|<4d$, we have $p_4\in R$ for each $R\in\mathcal{R}$. Since $p_5\in \Delta_{p_1,p_4,q}$ and rectangles are convex, each of these rectangles contains $p_5$. But since $\operatorname{dist}(\ell_0, p_5)\gtrapprox 1$, we have $\#\mathcal{R}\lessapprox 1$. In particular, we can select a rectangle $R\in\mathcal{R}$ with $\#(R\cap Q)\gtrapprox\#Q$. Fix this choice of $R$ and define $Q_1=Q\cap R.$

If $l\in L$ and $q\in Q_1$ with $(p_1,l),(q,l)\in I$, then since $\operatorname{dist}(p_1,q)\approx 1$ and since $p_1$ and $q$ are both contained in $R$, there is a constant $C_0=C_0(\alpha)$ (which depends on the implicit constants in the above quasi-inequalities) so that if $\delta>0$ is selected sufficiently small, then 
\begin{equation}\label{angleLR}
\angle(l,\operatorname{dir}(R))\leq \delta^{-C_0\eps}d.
\end{equation}
The same holds true with $p_2$ and $p_4$ in place of $p_1$. Crucially, the same also holds true with $p_5$ in place of $p_1$---since $p_5\in\Delta_{p_1,p_4,q}$, we must have $p_5\in R$. Define 
\[
L(p_1)=\{l\in I(p_1)\colon  \angle(l,\operatorname{dir}(R))\leq \delta^{-C_0\eps}d\}.
\]
Recall that $l\in I(p_1)$ implies that $p_1\in N_{\delta}(l)$. Abusing notation (but hopefully not the sensibilities of our readers), we shall slightly perturb each line $l\in L(p_1)$ so that $p_1\in l$. If $q\in Q_1$ was a point with $q\in N_\delta(l)$ prior to this perturbation, then $q\in N_{2\delta}(l)$ after the perturbation; this is harmless for our arguments. Define $L(p_2),L(p_4)$, and $L(p_5)$ similarly. By Properties (\ref{itemA}) and (\ref{itemD}), the cardinality of these sets is $\lessapprox \delta^{-\alpha} (\#L)(\#P)^{-1}.$

Observe that for each $q\in Q_1$, there are $\approx 1$ lines $l_1\in L(p_1)$ (resp. $l_2\in L(p_2)$ and $l_4\in L(p_4)$) that are incident to $q$. The converse, however is not true; the intersection of the $\delta$-neighborhoods of three lines $l_1,l_2,l_4$ is contained in a rectangle of dimensions roughly $\delta\times \delta/d$; the long axis of this rectangle makes an angle $\lesssim d$ with the long axis of $R$ (since both of these rectangles have eccentricity about $d^{-1}$, the two rectangles are as ``parallel'' as possible). By Property (\ref{itemB}), there could be as many as $\lessapprox  d^{-\alpha}$ points $q\in Q_1$ in this rectangle that are incident to $l_1,l_2$ and $l_4$. Let $Q_2\subset Q_1$, with 
\begin{equation}\label{lowerBdOnQ}
\#Q_2\gtrapprox   d^{\alpha}(\#Q_1)\gtrapprox d^{-\alpha}\delta^{-10\alpha}(\#L)^4(\#P)^{-8},
\end{equation}
so that each rectangle with dimensions $\delta\times \delta/d$ that is parallel to $R$ contains at most one point from $Q_2$.

Let $\tilde\delta = \delta/d$. Let $\tilde T\colon\RR^2\to\RR^2$ be the composition of a translation, rotation, and dilation that first sends $p_1$ to $(0,0)$, then sends $p_4$ to a point on the $x$-axis, and finally stretches $\RR^2$ in the $x$-direction, so that the image of $p_4$ becomes $(1,0)$. The image of $R$ under $\tilde T$ is a rectangle of dimensions $\approx 1 \times 1$. Let $\tilde p_i$ and $\tilde L(p_i)$, $i=1,2,4,5$ be the image of $p_i$ and $L(p_i)$, respectively, under $\tilde T$. Let $\tilde Q_2$ be the image of $Q_2$ under $T$.

We have that $\tilde p_1 = (0,0)$, $\tilde p_4=(1,0)$, and $\tilde p_2$ is contained in the $\delta$-neighborhood of the $x$-axis and has distance $\approx 1$ from $\tilde p_1$ and $\tilde p_4$. Abusing notation, we will replace $\tilde p_2$ with a point on the $x$-axis, and modify the set $\tilde L(p_2)$ accordingly.

The purpose of this transformation is that now the points in $\tilde Q_2$ are $\tilde\delta$-separated, and each $q\in \tilde Q_2$ satisfies
\begin{equation}\label{qFarFromBadLines}
\operatorname{dist}(q, x\textrm{-axis})\approx 1,\quad \operatorname{dist}(q,  \overline{\tilde p_1\tilde p_5})\approx 1,\quad \operatorname{dist}(q,  \overline{\tilde p_4\tilde p_5})\approx 1.
\end{equation}

Finally, the lines in $\tilde L(p_1)$ satisfy a re-scaled version of the non-concentration estimate, Property (\ref{nonConcentrationLinesThroughAPointItem}). For each unit vector $v\in S^1$ and each $r\geq \tilde\delta$, we have
\begin{equation}\label{nonConcentrationLinesThroughAPointRescaled}
\# \{ l \in \tilde L(p_1), \angle(l, v) \leq r \} \lessapprox (dr/\delta)^\alpha=(r/\tilde\delta)^{\alpha},
\end{equation}
and similarly for $\tilde L(p_2),$ $\tilde L(p_4),$ $\tilde L(p_5).$   

\paragraph{Applying a projective transformation.}
Next, we apply a projective transformation $T^\dag$ that sends the $x$-axis to the line at infinity. We choose this transformation so that lines containing $(0,0)$ (and in particular, lines in $\tilde L(p_1)$) are mapped to vertical lines; lines containing $(1,0)$ (and in particular, lines in $\tilde L(p_4)$) are mapped to horizontal lines; and lines containing $\tilde p_2$ (and in particular, lines in $\tilde L(p_2)$) are mapped to lines with slope $-1$.

Recall that $\tilde p_1,$ $\tilde p_2$, and $\tilde p_4$ have separation $\approx 1$; $\tilde Q_2$ is contained in a ball centered at the origin of radius $\approx 1$; and
by \eqref{qFarFromBadLines}, $\operatorname{dist}(\tilde Q_2,\ x\textrm{-axis})\approx 1$. Thus $T^\dag$ does minimal damage to $\tilde Q_2$---the set remains contained in a ball centered at the origin of radius $\approx 1$; it is $\approx\tilde\delta$ separated; and if $l\in\tilde L(p_1)$ was $\tilde\delta$-incident to some $q\in\tilde Q_2$, then the image of $l$ remains $\approx\tilde\delta$-incident to the image of $q$.

Recall from \eqref{defnD} that $\operatorname{dist}(p_5, \overline{p_1p_4})\approx 1$, and thus $\operatorname{dist}(\tilde p_5, x\textrm{-axis})\approx 1$. This means that the image of $\tilde p_5$ under $T^\dag$ has distance $\lessapprox 1$ from the origin, so we can apply a translation and suppose that the image of $\tilde p_5$ is $(0,0)$. After this translation we still have that the image of $\tilde Q_2$ is contained in a ball centered at the origin of radius $\approx 1$. Abusing notation, we will continue to refer to this transformation as $T^\dag$. 

Let $L(p_1)^\dag$, $L(p_2)^\dag$, $L(p_4)^\dag$, $L(p_5)^\dag$ denote the images of $\tilde L(p_1)$, $\tilde L(p_2)$, $\tilde L(p_4)$, $\tilde L(p_5)$ under $T^\dag$. $T^\dag$ maps the line $\overline{\tilde p_1,\tilde p_5}$ to the $y$-axis, and the line $\overline{\tilde p_4,\tilde p_5}$ to the $x$-axis. Thus by \eqref{qFarFromBadLines}, each $q^\dag \in T^\dag(\tilde Q_2)$ has distance $\approx 1$ from the $x$ and $y$ axes. For convenience, we will suppose that each $q^\dag \in T^\dag(\tilde Q_2)$ is contained in the first quadrant; this can always be achieved by applying a suitable reflection and refining $\tilde Q_2$ by a factor of at most 4. Next, by dyadic pigeonholing, we can suppose there are numbers $M_x,M_y\approx 1$ so that after refining $\tilde Q_2$ by at most $|\log\delta|^{-1}$, we have that $T^\dag(\tilde Q_2)$ is contained in $[M_x, 2M_x]\times [M_y, 2M_y]$. Thus if we compose $T^\dag$ with the anisotropic scaling $S(x,y)=(M_x^{-1}x, M_y^{-1}y)$ and define $Q^\dag = S\circ T^\dag(\tilde Q \cap [M_x, 2M_x]\times [M_y, 2M_y])$, then $Q^\dag\subset [1,2]^2$. Abusing notation slightly, we will replace $L(p_1)^\dag$, $L(p_2)^\dag$, $L(p_4)^\dag$, $L(p_5)^\dag$ with their image under $S$. 

Finally, we will choose a value of $\delta^\dag\approx\tilde\delta$ so that after further refining $Q^\dag$ by a $\approx 1$ factor, we can ensure that $Q^\dag$ is $\delta^\dag$-separated, and if $l\in\tilde L(p_1)$ was $\tilde\delta$-incident to some $q\in\mathcal{Q}_2$, then the image of $l$ will be $\delta^\dag$-incident to the image of $q.$

Recall that the lines in $L(p_1)^\dag$ are parallel to the $y$-axis; let $X\subset\RR$ be the set of $x$-intercepts of these lines. The lines in $L(p_4)^\dag$ are parallel to the $x$-axis; let $Y\subset\RR$ be the set of $y$-intercepts of these lines. The lines in $L(p_2)^\dag$ are of the form $x+y = t$; let $Z\subset\RR^2$ be the set of $t$-values for such lines. The lines in $L(p_5)\dag$ are of the form $y = mx$; let $W$ be the set of slopes (i.e. $m$ values) of these lines. 

The sets $X, Y, Z, W$ are $\approx \delta^\dag$ separated. Since each $q^\dag\in Q_2^\dag$ is contained in $[1,2]^2$, each line $l\in L(p_1)^\dag$ that is incident to at least one point of $Q_2^\dag$ must have $x$-intercept in $[1,2]$. Similarly, each line $l\in L(p_4)^\dag$ that is incident to at least one point of $Q_2^\dag$ must have $y$-intercept in $[1,2]$. Thus we can suppose that $X,Y\subset[1,2]$ (i.e. we can discard any elements of $X$ or $Y$ for which this condition fails, since such lines cannot be incident to any points in $Q_2^\dag$).

The angular non-concentration \eqref{nonConcentrationLinesThroughAPointRescaled} becomes a spatial non-concentration condition on the sets $X$ and $Y$. Specifically, for each interval $J\subset\RR$ we have
\begin{equation}
\begin{split}
&\#(X \cap J) \lessapprox (|J|/\delta^\dag)^{\alpha},\\
&\#(Y \cap J) \lessapprox (|J|/\delta^\dag)^{\alpha}.
\end{split}
\end{equation}
Note that the ``$\lessapprox$'' in the above inequalities represents a multiplicative factor of the form $\delta^{-C\eps}$, not $(\delta^\dag)^{-C\eps}$.
The our bounds on the size of $L(p_i)$ allow us to estimate
\begin{equation}\label{upperBdOnXandY}
\begin{split}
&\#X\lessapprox \delta^{-\alpha}(\#L)(\#P)^{-1},\\
&\#Y\lessapprox \delta^{-\alpha}(\#L)(\#P)^{-1},\\
&\#Z\lessapprox \delta^{-\alpha}(\#L)(\#P)^{-1},\\
&\#W\lessapprox \delta^{-\alpha}(\#L)(\#P)^{-1}.
\end{split}
\end{equation}

Recall that $X_{\delta^\dag} = N_{\delta^\dag}(X)\cap (\delta^\dag \ZZ)$. Let $E\subset X_{\delta^\dag}\times Y_{\delta^\dag}$, with $(x,y)\in E$ if there are points $x'\in X$, $y'\in Y$ with $|x-x'|\leq\delta^\dag,$ $|y-y'|\leq\delta^\dag,$ so that the corresponding lines $l_1\in L(p_1)^\dag$, $l_4\in L(p_4)^\dag$ with intercepts $x'$ and $y'$ are $\delta^\dag$-incident to a common point of $Q_2^\dag$. By \eqref{lowerBdOnQ},  
\[
\# E\gtrapprox  d^{-\alpha}\delta^{-10\alpha}(\#L)^4(\#P)^{-8}. 
\]

Furthermore, if $(x,y)\in E$, then the corresponding point $q\in\tilde Q_2^\dag$ is incident to a line of the form $x+y=t$, where $|t-(x-y)|\leq 4\delta^\dag$. This means that $x-y$ is an element of $\delta^\dag\ZZ$ that is contained in $N_{4\delta^\dag}(Z)$. But since $Z$ is $\approx \delta^\dag$ separated, by \eqref{upperBdOnXandY}  we have that
\[
\#(X_{\delta^\dag} \overset{E}{-} Y_{\delta^\dag})\leq 4\mathcal{E}_{\delta^\dag}(Z)\lessapprox\#Z \lessapprox \delta^{-\alpha}(\#L)(\#P)^{-1}.
\]

\paragraph{Bounding the sum-set.}
Next, we will use the Balog-Szemer\'edi-Gowers theorem to find a large set $A_0\subset X_{\delta^\dag}$ whose sum-set is small. By Lemma~\ref{bsg} there exists a set $A_0\subset X_{\delta^\dag}$ with
\begin{equation}\label{boundOnA0}
\begin{split}
\# A_0 &\gtrapprox \frac{d^{-\alpha}\delta^{-10\alpha}(\#L)^4(\#P)^{-8} }
{\delta^{-\alpha}(\#L)(\#P)^{-1} }\\
&= \Big(d^{-2\alpha}\delta^{-8\alpha}(\#L)^3(\#P)^{-7}\Big)(\delta^\dag)^{-\alpha},
\end{split}
\end{equation}
and

\begin{equation}\label{EDeltaAA}
\begin{split}
\#(A_0-A_0) &\lessapprox 
\frac{\big(\delta^{-\alpha}(\#L)(\#P)^{-1} \big)^{12}}
{\big(d^{-\alpha}\delta^{-10\alpha}(\#L)^4(\#P)^{-8}  \big)^6} (\#A_0)\\
&= \Big(d^{6\alpha}\delta^{48\alpha}(\#L)^{-12}(\#P)^{36}\Big) (\#A_0).
\end{split}
\end{equation}
For the latter inequality here we made use of the bound $\#E\lesssim (\#A_0)(\#Y)$, which is ensured by Lemma~\ref{bsg}; this is needed in order to obtain an estimate of the form $\#(A_0-A_0)\leq K\#A_0$, giving control on the growth of $A_0$ under subtraction relative to the cardinality of $A_0$ itself.

Next, we will find a large set $A\subset A_0$, so that $A-A$ has small covering number at every scale (c.f.~Item \ref{itm:3} from Proposition \ref{modifiedGKZ}). Our arguments are similar to the proof of Lemma 2.1 from \cite{gkz}. Let $k=1/\eps$ and select $T$ so that $2^{kT}\leq (\delta^\dag)^{-1}\leq 2^{k(T+1)}$. Apply Lemma \ref{moranRegularLem} to $A_0$; let $A\subset A_0$ and $(N_0,\ldots,N_{k-1})$ be the output from that lemma. $A$ continues to satisfy \eqref{boundOnA0}. We will show that for each $\delta^\dag\leq\rho\leq 1,$ we have 
\begin{equation}\label{rhoCoveringAA}
\mathcal{E}_{\rho}(A-A)\lessapprox \Big(d^{6\alpha}\delta^{48\alpha}(\#L)^{-12}(\#P)^{36}\Big) \mathcal{E}_{\rho}(A).
\end{equation}
First, we will verify \eqref{rhoCoveringAA} for $\rho$ of the form $2^{-jT},\ 1\leq j\leq k$. Since each $2^T$-adic interval of length $2^{-jT}$ contains $\prod_{i=j}^{k-1}N_i$ points from $A$, we have 
\[
\#(A-A) = \mathcal{E}_{\delta}(A-A)\gtrsim \Big(\prod_{i=j}^{k-1}N_i\Big)\mathcal{E}_{\rho}(A-A).
\]
Thus
\begin{equation*}
%\begin{split}
\mathcal{E}_{\rho}(A-A) \lesssim \Big(\prod_{i=j}^{k-1}N_i\Big)^{-1}\#(A-A)\lessapprox \Big(\prod_{i=j}^{k-1}N_i\Big)^{-1} \Big(d^{6\alpha}\delta^{48\alpha}(\#L)^{-12}(\#P)^{36}\Big) (\#A_0).
\end{equation*}
The estimate \eqref{rhoCoveringAA} now follows from the inequality $\#A\approx\#A_0$, and the observation that  $\mathcal{E}_{\rho}(A)=\big(\prod_{i=j}^{k-1}N_i\big)^{-1}(\#A)$.

Next, to verify \eqref{rhoCoveringAA} for arbitrary $\rho,$ select $1\leq j\leq k$ with $2^{-jT}\leq \rho\leq 2^{-(j-1)T}$ and let $\rho'=2^{-jT}$. Then \eqref{rhoCoveringAA} follows from the observation that $\mathcal{E}_{\rho}(A-A) \leq \mathcal{E}_{\rho'}(A-A),$ and $\mathcal{E}_{\rho'}(A)\approx \mathcal{E}_{\rho}(A)$.

\paragraph{Bounding the multiplicative energy.}
We still have that every point in $E \cap (A\times Y_{\delta^\dag})$ is $\delta^\dag$-incident to a line of the form $y = mx$, with $m\in W$. Thus by Cauchy-Schwarz, 
\[
\#\{(p,p')\in E \cap (A\times Y)\colon p,p'\ \textrm{are $\delta^\dag$-incident to a common line}\ y=mx,\ m\in W \} \gtrapprox 
\frac{\big(\#(E \cap (A\times Y))\big)^2}
{\#W}.
\]
Since $A,Y\subset [1,2]$, each pair $(p,p')$ in the above set contributes to the multiplicative energy $E^{\delta^\dag}_{\times}(A,Y^{-1})$, and thus we have
\[
E^{\delta^\dag}_{\times}(A,Y)^{-1}\gtrapprox 
\frac{\big(\#(E \cap (A\times Y))\big)^2}
{\#W}.
\]
Since $Y\subset[1,2]$ is $\approx \delta^\dag$-separated, we have $E^{\delta^\dag}_{\times}(Y^{-1},Y^{-1})\lessapprox (\#Y)^3$. By Lemma~\ref{energyCauchySchwartz}, we have
\[
E^{\delta^\dag}_{\times}(A,A)\geq 
\frac{\big(E^{\delta^\dag}_{\times}(A,Y^{-1})\big)^2}
{E^{\delta^\dag}_{\times}(Y^{-1},Y^{-1})}
\gtrapprox \frac{\big(\#(E \cap (A\times Y))\big)^4}
{(\#W)^2(\#Y)^3}.
\]
Recall that at the end of the Fixing four special points part of the argument, we applied dyadic pigeonholing to ensure that each line incident to $p_1$ was incident to the same number of points in $Q$, up to a factor of two (this was how we obtained the set $Q$ from $Q_0$). As a consequence of this property, each line that was used to define an element of $X$ is incident to $\gtrapprox (\#E)/(\#X)$ points that are in turn incident to lines that were used to define the set $Y$. Thus, we have
\begin{equation}\label{consequenceOfEvenOnLine}
\#\big(E \cap (A\times Y)\big)\gtrapprox (\# E)\frac{\# A}{\# X},
\end{equation}
and we therefore get
\begin{equation}
    \begin{split}
    E^{\delta^\dag}_{\times}(A,A) &\gtrapprox \frac{(\#E)^4(\#A)^4}
{(\#W)^2(\#X)^4(\#Y)^3}\\
&\gtrapprox \frac{(\#E)^5}
{(\#W)^2(\#X)^4(\#Y)^4} (\#A)^3\\
&\gtrapprox  \big(d^{-5\alpha} \delta^{-40\alpha}(\#L)^{10}(\#P)^{-30}\big)(\#A)^3.    
    \end{split}
\end{equation}
Note that here we used the bound $\#E \lessapprox (\#A)(\#Y)$, which holds since Lemma~\ref{bsg} ensures that $\#E \lessapprox (\#A_0)(\#Y)$ and Lemma~\ref{moranRegularLem} ensures that $\#A\gtrapprox \#A_0$.

\paragraph{Applying discretized sum-product.}
To summarize, we have a $A\subset [1,2]\cap(\delta^\dag \ZZ)$, with
\begin{equation}\label{lowerBdOnA}
\# A \gtrapprox \Big(d^{2\alpha}\delta^{8\alpha}(\#L)^{-3}(\#P)^{7}\Big)^{-1} (\delta^\dag)^{-\alpha},
\end{equation}
which satisfies the non-concentration condition
\begin{equation}\label{nonConcentrationA}
\#(A\cap J)\lessapprox |J|^{\alpha} (\delta^\dag)^{-\alpha},
\end{equation}
and
\begin{equation}\label{differenceAndEnergyBoundsA}
\begin{split}
\mathcal{E}_{\rho}(A-A)&\lessapprox
\Big(d^{6\alpha}\delta^{48\alpha}(\#L)^{-12}(\#P)^{36}\Big) \mathcal{E}_{\rho}(A),\quad \delta^\dag\leq\rho\leq 1,\\
E^{\delta^\dag}_{\times}(A,A) & \gtrapprox  \Big(d^{5\alpha} \delta^{40\alpha}(\#L)^{-10}(\#P)^{30}\Big)^{-1}(\#A)^3.
\end{split}
\end{equation}

Comparing this with Proposition \ref{modifiedGKZ} and using the bounds \eqref{lowerBdOnA}, \eqref{nonConcentrationA}, and \eqref{differenceAndEnergyBoundsA}, we conclude that

\begin{equation*}
\begin{split}
& \Big(d^{2\alpha}\delta^{8\alpha}(\#L)^{-3}(\#P)^{7}\Big)^6\\
&\cdot\Big(d^{6\alpha}\delta^{48\alpha}(\#L)^{-12}(\#P)^{36}\Big)^{18+8\alpha}\Big(d^{5\alpha} \delta^{40\alpha}(\#L)^{-10}(\#P)^{30} \Big)^{8+4\alpha} \\
&= d^{4\alpha(40+17\alpha)} \delta^{16\alpha(77+34\alpha)} (\#L)^{-2(157+68\alpha)} (\#P)^{6(155+68\alpha)}\\
&= (\delta^{2\alpha} \#P)^{6(155+68\alpha)} (\delta^{2\alpha} \#L)^{-2(157+68\alpha)} d^{4\alpha(40+17\alpha)}\\
&\gtrapprox (\delta^\dag)^{-\alpha(1-\alpha)},
\end{split}
\end{equation*}
i.e.
\[
(\delta^{2\alpha} \#P)^{6(155+68\alpha)} \gtrapprox  d^{-3\alpha(53+23\alpha)}  \delta^{-\alpha(1-\alpha)}(\delta^{2\alpha}\#L)^{2(157+68\alpha)}.
\]
It is now clear that the worst case occurs when $d$ is as large as possible (i.e. $d\sim 1$). We conclude that
\[
\#P \gtrapprox \delta^{-2\alpha-c(\alpha)}(\delta^{2\alpha}\#L)^{\frac{157+68\alpha}{3(155+68\alpha)}},\quad c(\alpha) = \frac{\alpha(1-\alpha)}{6(155+68\alpha)}.
\]
The final step is to verify that for all $\alpha\in(0,1)$, $\frac{157+68\alpha}{3(155+68\alpha)}\leq 1$.
\end{proof}

\begin{remark}\label{discussionOfLossesRemark}
When $c=1/2$, the discretized sum-product theorem (Proposition \ref{modifiedGKZ}) gives a gain of roughly $\delta^{-1/100}$. The final gain Proposition \ref{discretizedFurstenbergWithNonConcentrationProp} is roughly 50 times, or two orders of magnitude, worse than this. Roughly speaking, one order of magnitude comes from repeated pigeonholing and use of Cauchy-Schwarz, while the other order of magnitude comes from the use of Balog-Szemer\'edi-Gowers. While some further optimizations are likely possible, it appears that such optimizations would yield rather modest gains. In particular, Proposition \ref{modifiedGKZ} requires that the difference set $A-A$ be small (rather than merely requiring that $E_{+}(A,A)$ be large), so it appears difficult to avoid using the Balog-Szemer\'edi-Gowers theorem). While a more efficient sequence of pigeonholing and Cauchy-Schwarz arguments might improve the final bound, since the current argument only loses (roughly) one order of magnitude at this step, any improvement here would at best yield a modest improvement in the final bound. 
\end{remark}

\section{From Proposition \ref{discretizedFurstenbergWithNonConcentrationProp} to discretized Furstenberg}\label{reductionSection}
Our goal in this section is to prove that Proposition \ref{discretizedFurstenbergWithNonConcentrationProp} implies Proposition \ref{discretizedFurstenbergProp}, and hence Theorem \ref{dimFurstenberg}. First, we will briefly describe the differences between Propositions \ref{discretizedFurstenbergWithNonConcentrationProp} and \ref{discretizedFurstenbergProp}. Both are incidence theorems about collections of points and lines, but Proposition \ref{discretizedFurstenbergWithNonConcentrationProp} requires two additional non-concentration conditions. 

Property (\ref{robustTransversalityOfLinesThruPtItem}) is sometimes called ``robust transversality,'' and it asserts that for each $p\in P$, a typical pair of lines incident to $p$ point in well-separated directions. With this condition, Proposition \ref{discretizedFurstenbergWithNonConcentrationProp} can also be viewed as a bilinear incidence theorem. Thus our reduction from Proposition \ref{discretizedFurstenbergWithNonConcentrationProp} to \ref{discretizedFurstenbergProp} will employ ``bilinear to linear'' type arguments, which were first used by Tao, Vargas, and Vega in the context of the Kakeya and Restriction problems \cite{tvv}. A prototype for the specific arguments used here can be found in \cite[\S 2.3]{gz}. See also \cite{TaoBlog} for a similar argument in a slightly different setting.

Property (\ref{nonConcentrationLinesThroughAPointItem}) is also an assertion about the distribution of directions $\{\operatorname{dir}(l)\colon (p,l)\in I\}$. This property imposes stronger non-concentration conditions at small scales (specifically at length scales between $\delta$ and $(\delta^{-2\alpha}/\#P)^{1/\alpha}$), but it says nothing about non-concentration at larger scales. To the best of our knowledge, the arguments used to obtain the assertion in Property (\ref{nonConcentrationLinesThroughAPointItem}) are novel, and we hope that similar arguments might be helpful for related problems. %Our arguments for this reduction require that for each point $p$, the distribution of directions $\{\operatorname{dir}(l)\colon (p,l)\in I\}$ forms a Moran-regular tree. %We will discuss Moran regularity below. 

\begin{proof}[Proof of Proposition \ref{discretizedFurstenbergProp}]
Fix $0<\alpha<1$ and $c<c(\alpha)=\frac{\alpha(1-\alpha)}{6(155+68\alpha)}$. We will show that if the parameters $\eps>0$ and $C$ are chosen appropriately, then for all $\delta>0$ sufficiently small, if $E\supset \bigcup_{l\in L_0}P_l$ is a discretized $(\alpha,2\alpha)$-Furstenberg set (with implicit parameters $\eps$ and $C$), then
\begin{equation}\label{coveringNumberOfE}
\mathcal{E}_\delta(E) \geq \delta^{-2\alpha -c}.
\end{equation}
In particular, \eqref{coveringNumberOfE} will follow from the estimate 
\begin{equation}\label{coveringNumberOfEApprox}
\mathcal{E}_\delta\Big( \bigcup_{l\in L_0}P_l \Big) \gtrapprox \delta^{-2\alpha -c(\alpha)}.
\end{equation}

%%%%%%%%%%%%%%%%%%%
%%%%%%%%%%%%%%%%%%%%
%%%%%%%%%%%%%%%%%%%%
\paragraph{Constructing the incidence arrangement.}
Our goal is to construct a $\delta$-separated set $P\subset\bigcup_{l\in L_0}P_l$, and an incidence relation $I\subset P\times L$ that satisfies the hypotheses of Proposition \ref{discretizedFurstenbergWithNonConcentrationProp}. While it is tempting to simply choose $P$ to be a maximal $\delta$-separated subset of $\bigcup_{l\in L_0}P_l$, it is possible that a set constructed in this fashion will have the following undesirable property: let $R$ be a rectangle of dimensions $r\times\delta$, so that many lines $l\in L_0$ intersect $R$ in a line-segment of length roughly $r$; while Item (\ref{secondItemDiscretizedFurstenbergDefn}) from Definition \ref{discretizedFurstDef} guarantees that each set $R\cap P_l$ has cardinality $\lessapprox (r/\delta)^\alpha$, the union $R\cap \bigcup P_l$ might be much larger. For technical reasons that will become apparent later in the proof, we would like to avoid this type of incidence arrangement. 

To avoid this problem, we will use a stopping-time type argument to construct an incidence arrangement $(P,L,I)$ so that for each rectangle $R$ of this type, only a few lines from $L_0$ that have large intersection with $R$ contribute points to $R$. % See \eqref{fewContributingLines} below for details.

Let $k = \lfloor 1/\eps\rfloor$ and select $T$ so that $2^{-kT}\leq\delta/6< 2^{-k(T-1)}$. Refining each set $\{P_l\}_{l\in L_0}$ by a factor of at most $\delta^{-k-3}$, we can suppose that each set $P_l$ is discretized at scale $2^{-kT}$. Apply Corollary \ref{moranRegularCor} to the collection of sets $\{P_l\}_{l\in L_0}$, and let $L_1\subset L_0$; $\{P_l'\}_{l\in L_1};$ and $(N_0,\ldots,N_{k-1})$ be the output from the corollary. We have

\begin{equation}\label{lotsOfPointsOnLines}
\sum_{l\in L_1}\#P_l' \gtrapprox \sum_{l\in L}\#P_l \gtrapprox \delta^{-3\alpha}.
\end{equation}

Enumerate the elements of $L_1$ as $l_1,\ldots,l_M$, $M\approx\delta^{-2\alpha}$. Will will construct a point set $P_1\subset\bigcup_{l\in L}P_l$ and an incidence relation $I_1\subset P_1\times L_1$ as follows.

We begin by setting $P_1=\emptyset$ and $I_1 = \emptyset$. Let $m=1$, and add all of the points in $P_{l_1}'$ to $P_1$; add all of the incidences $\{ (p,l_1)\colon p\in P_{l_1}'\}$ to $I_1$. For each $2^T$-adic square $S$ for which $P_{l_1}'\cap S\neq \emptyset$, we say that $l_1$ has \emph{contributed points to square $S$}. We have now ``processed'' the line $l_1$.

Suppose that we have processed the lines $l_1,\ldots,l_{m-1}$. We shall process the line $l_m$ as follows. To begin, we declare all $2^T$-adic squares inside $[0,1]^2$ to be ``undominated.'' We perform the following procedure for each $j=0,\ldots, k$, starting with $j=0$: consider each $2^T$-adic square $S\subset[0,1)^2$ of side-length $2^{-jT}$ that is not contained in a (larger) dominated square. 
\begin{itemize}
\item If $S\cap {P'_{l_m}}=\emptyset$, we ignore $S$, and it remains undominated (squares of this type will not be relevant for successive steps of the algorithm). 
\item If instead $P'_{l_m}\cap S\neq\emptyset$, and hence  $\#(P'_{l_m}\cap S)=\prod_{i=j}^{k-1}N_i$, then either (A) or (B) must hold:
	\begin{itemize}
		\item[(A):] there is at least one line $l_n,\ n<m$ so that: 
			\begin{itemize}
			\item[(i):] $l_n$ has contributed points to $S$, 
			\item [(ii):] $\angle(l_m, l_n)\leq (\frac{\pi}{2})2^{(j-k)T}$, 
			\item [(iii):] $\operatorname{dist}(l_m\cap S, l_n\cap S)\leq (\sqrt 2)2^{-kT}.$
			\end{itemize}
		\item[(B):] no such line exists.
	\end{itemize}
	If Option (A) occurs, select a line $l_n$ with this property; add the pairs $\{(p,l_m)\colon p\in S\cap P'_{l_n}\}$ to $I_1$; and mark the square $S$ as dominated. We will say that ``$l_m$ is dominated by $l_n$ on square $S$.'' Note that the conditions in Option (A) imply that $l_n\cap S \subset N_{\delta}(l_m)$ (indeed, there is a point in $S$ where $l_m$ and $l_n$ have separation at most $(\sqrt 2) 2^{-kT}\leq \delta/4$, and the condition on $\angle(l_m,l_n)$ then ensures that $l_n\cap S \subset N_{\delta}(l_m)$. Thus each of the incidences $(p,l_m)$ added in this step satisfy $p\in N_{\delta}(l_m)$. If Option (B) occurs, the square $S$ remains undominated. 
\end{itemize}
Repeat this process for $j=1,2,\ldots,k$. We say a $2^T$-adic square $S$ is a \emph{purely undominated square} if $S$ does not contain a dominated square and is not contained in a dominated square. Moreover, we say a purely undominated square $S$ is \emph{maximal purely undominated} every $2^T$-adic square $S'\supsetneq S$ \emph{does} contain a dominated square. Each purely undominated square $S$ is contained in a unique maximal purely undominated square. For each maximal purely undominated square $S$ for which $S\cap P_{l_m}'\neq\emptyset$, add the points in $S\cap P_{l_m}'$ to $P_1$, and add the corresponding incidences $(p,l_m)$ to $I_1$. For each purely undominated $2^T$-adic square $S$ (maximal or not) for which $P_{l_m}'\cap S\neq \emptyset$, we say that $l_m$ has contributed points to square $S$. 

After the above steps have been completed, we have now processed the line $l_m$. Observe that at this point, the incidence  arrangement $(P_1,L_1,I_1)$ has the following properties:
\begin{itemize}
\item If $(p,l_m)\in I_1$, then $p\in N_{\delta}(l_m)$
\item For each $2^T$-adic square $S$, there is a line $l_n$ (both $l_n=l_m$ and $l_n\neq l_m$ are possible) so that
\begin{equation}\label{pinP1VsPinPl}
\#(S\cap I_1(l_m)) =\#(S \cap P_{l_n}').
\end{equation}
\end{itemize}
As a consequence of \eqref{pinP1VsPinPl} when $S=[0,1)^2$, we have $\# I_1(l_m)\gtrapprox \delta^{-\alpha}.$ Furthermore, 
\begin{equation}\label{propertyTwoOfI1}
\#(B(x,r)\cap I_1(l_m))\lessapprox (r/\delta)^{\alpha}\quad\textrm{for all $\delta\leq r$ and all balls}\ B(x,r).
\end{equation}
To verify \eqref{propertyTwoOfI1}, fix a ball $B(x,r)$ and choose $j$ so that $2^{-jT}\leq r< 2^{-(j-1)T}$. Then $B(x,r)$ intersects $\approx 1$ squares of side-length $2^{-jT}$, so \eqref{propertyTwoOfI1} follows from \eqref{pinP1VsPinPl}, plus the fact that $P_{l_n}'$ is a $(\delta,\alpha)_1$-set.

As additional lines are processed, we will add more points to $P_1$ and more incidences to $I_1$, but this will not disrupt the properties of $l_m$ described above.

%%%%%%
%%%%%%

After all of the lines in $L_1$ have been processed in the above fashion, the incidence arrangement $(P_1, L_1, I_1)$ satisfies  Properties (\ref{itemB}) (with $V=1$) and (\ref{itemC}) from Proposition \ref{discretizedFurstenbergWithNonConcentrationProp}, as well as the estimate $\#I_3\approx\delta^{-3\alpha}$. We claim that $P_1$ is discretized at scale $2^{-kT}$, and thus $\lessapprox 1$ points $p\in P_1$ can be contained in any ball of radius $\delta$. To prove this claim, note that if $S$ is a square of side-length $2^{-kT}$, then at most one line $l\in L_1$ can contribute points to $S$---if two such lines $l,l'$ both contributed points to $S$, then we must have that either $\angle(l,l')> \pi/2$ (which is impossible), or $\operatorname{dist}(l\cap S, l'\cap S)> (\sqrt{2})2^{-kT}=\operatorname{diameter}(S)$ (which is impossible). 

Since $P_1$ is discretized at scale $2^{-kT}$, in order to establish \eqref{coveringNumberOfEApprox} it suffices to prove that
\begin{equation}\label{necessaryBdP1}
\#P_1 \gtrapprox \delta^{-2\alpha -c(\alpha)}.
\end{equation}

Finally, we observe that if $p\in P_1$, if $1\leq j\leq k-1$, and if $S$ is a square of side-length $2^{-jT}$ that contains $p$, then there is exactly one line $l\in I_1(p)$ that contributes points to square $S$ (though note that other lines might contribute points to sub-squares of $S$ that do not contain $p$).

\paragraph{Non-concentration of lines through a point.}
Our next task is to extract a set of points $P_2\subset P_1$ and a set of incidences $I_2\subset I_1$ so that for each $p\in P_2$, the distribution of directions of lines incident to $p$ satisfies Property (\ref{nonConcentrationLinesThroughAPointItem}). 

For each $p\in P_1$, we will identify the set $I_1(p)$ with a subset of $[0,1]$ by identifying the line $y=mx+b$ with the point $m$; call this set $A_p$. Since $L_0$ is $\delta$-separated, at most 16 points in $A_p$ can intersect any interval of length $2^{-kT}\leq \delta/6$ (indeed, the set of such points in $A_p$ corresponds to lines in $L_1$ whose image under $\iota$ is contained in a ball of radius $2\delta$). In particular, by refining each set $A_p$ by a factor of 16, we can suppose that $A_p$ is $2^{-kT}$-discretized. By removing an additional point if necessary, we can also suppose that $A_p\subset[0,1)$.

Apply Corollary \ref{moranRegularCor} to the collection of sets $\{A_p\}_{p\in P_1}$, and let $P_2\subset P_1$; $\{A_p'\}_{p\in P_2};$ $(M_0,\ldots,M_{k-1})$ be the output from the corollary. Each set $A_p'$ has cardinality $\prod_{j=0}^{k-1}M_j$, and
\begin{equation}\label{mostIncidencesPreserved}
(\#P_2)\prod_{j=0}^{k-1}M_j = \sum_{p\in P_2}\#A_p' \approx \sum_{p\in P_1}\#A_p\approx \#I_1\approx\delta^{-3\alpha}.
\end{equation}
% and since each set $A_p'$ has the same cardinality, we conclude that $\#A_p'\approx \#I_1(\#P_2)^{-1}\approx \delta^{-3\alpha}(\#P_2)^{-1}$ for each $p\in P_2$. 
Define $I_2$ to consist of those incidences $(p,l)\in P_2\times L_1$ such that the slope of $l$ is in $A_p'.$ %By \eqref{mostIncidencesPreserved} we have $\#I_2\gtrapprox\delta^{-3\alpha}.$

Recall that if $0\leq j\leq k$ and if $J\subset[0,1)$ is an $2^T$-adic interval of side-length $2^{-jT}$ that meets $A_p$, then $\#(J\cap A_p')= \prod_{i=j}^{k-1}M_i$. Our next task is to control the size of this product. In theory, each $M_i$ could be as large as $2^{T}$. We will show that on average, each $M_i$ has size at most $2^{\alpha T}$. More precisely, we will prove that there exists a constant $C_0=C_0(\alpha)$ so that 
\begin{equation}\label{linesNotTooConcentrated}
\prod_{i=j}^{k-1}M_i \leq  \delta^{-C_0\eps} \big( 2^{jT}\delta \big)^{-\alpha}, \quad\textrm{for each}\ j = 0,\ldots,k-1.
\end{equation}
We will prove \eqref{linesNotTooConcentrated} by induction---the base case will be when $j$ is large (i.e.~$j$ is close to $k-1$, which corresponds to small length scales), and the induction step will go from $j$ to $j-1$.  For the base case, the trivial estimate $M_i\leq 2^T$ suffices: if $C_0\geq 2(1-\alpha)(k-j)$ then,
\[
\prod_{i=j}^{k-1}M_i\leq 2^{(k-j)T} =2^{(k-j)T(1-\alpha)}2^{(k-j)T\alpha}\leq\delta^{-C_0\eps}\big( 2^{jT}\delta \big)^{-\alpha},
\]
and so \eqref{linesNotTooConcentrated} holds.

Suppose now that \eqref{linesNotTooConcentrated} has been established for all $j'>j$; our task is to prove \eqref{linesNotTooConcentrated} for $j$. If $M_j\leq 2$ then this follows immediately from \eqref{linesNotTooConcentrated} for $j+1$. Suppose now that $M_j\geq 3$. 

Since $(\#I_2)/(\#L_1)\gtrapprox\delta^{-\alpha}$, there exists a line $l\in L$ with $\# I_2(l)\gtrapprox\delta^{-\alpha}$. Fix this choice of $l$.  Define $\rho = 2^{(j-k)T}$. By \eqref{propertyTwoOfI1}, each ball $B$ of radius $\delta^{1-2\eps}\rho^{-1}$ satisfies 
$
\# ( I_2(l) \cap B)\lessapprox  (\delta^{1-2\eps}\rho^{-1}/\delta)^\alpha\lessapprox \rho^{-\alpha},
$
and hence we can select a set $P'\subset I_2(l)$ that is $(\delta^{1-2\eps}\rho^{-1})$-separated, with $\#P'\gtrapprox(\delta/\rho)^{-\alpha}$.

Let $m$ be the slope of $l$. For each $p\in P'$, we have that $m\in A_p$, and hence the $2^T$-adic interval $J$ of length $\rho$ containing $m$ intersects $A_p$. Thus for each $p\in P',$ there are $\geq \prod_{i=j}^{k-1}M_i $ lines $l' \in I_2(p)$ with $|m-m(l')|\leq\rho$. On the other hand, the lines $l' \in I_2(p)$ with $|m-m(l')|\leq 2^{-T}\rho$ are contained in a union of at most two $2^T$-adic intervals of length $2^{-(j+1)T}$; each of these intervals contains $\prod_{i=j+1}^{k-1}M_i = M_j^{-1}\prod_{i=j}^{k-1}M_i\leq\frac{1}{3}\prod_{i=j}^{k-1}M_i$ points. 

Thus for each $p\in P'$ there are at most $\frac{2}{3}\prod_{i=j}^{k-1}M_i$ lines $l'\in I_2(p)$ with $|m-m(l')|\leq 2^{-T}\rho$. We conclude that 
\[
\# \{l' \in I_2(p)\colon 2^{-T}\rho\leq |m-m(l')|\leq \rho\} \geq \frac{1}{3}\prod_{i=j}^{k-1}M_i.
\]
Call the above set of lines $L(p)$. Since the points in $P'$ are $\delta^{1-2\eps}\rho^{-1}$ separated and contained in the $\delta$-neighborhood of $l$, if $p,p'\in P'$ are distinct, then the sets $L(p)$ and $L(p')$ are disjoint (indeed, a line $l'\in L(p)\cap L(p')$ must intersect $N_{\delta}(l)$ in an interval of length $\geq \delta^{1-2\eps}\rho^{-1}$, but this is impossible since the slopes of $l$ and $l'$ differ by at least $2^{-T}\rho\geq \delta^{2\eps}\rho)$. Thus
\[
\# \bigcup_{p\in P'}L(p) \gtrapprox (\#P')\prod_{i=j}^{k-1}M_i\gtrapprox (\rho/\delta)^{\alpha}\prod_{i=j}^{k-1}M_i.
\]
The lines in the above set satisfy $\operatorname{dist}(\iota(l),\iota(l'))\leq C_1\rho$, for some $C_1\lesssim 1$. Since $L_1$ is a $(\delta,2\alpha)_2$-set, we have  
\begin{equation}\label{reverseBallEstimate}
\prod_{i=j}^{k-1}M_i \lessapprox(\delta/\rho)^{\alpha}\#\big(B(x,C_1\rho)\cap \iota(L_1)\big)\lessapprox (\delta/\rho)^{\alpha}(\rho/\delta)^{2\alpha}=(\rho/\delta)^\alpha.
\end{equation}
Selecting $C_0$ sufficiently large based on the implicit constants in the above quasi-inequality (these constants in turn depend only on $\alpha$), we conclude that \eqref{linesNotTooConcentrated} holds for $j$. This completes the induction step.

By \eqref{linesNotTooConcentrated}, if $p\in P_2$ and $J\subset S^1$ is an interval whose image $J'\subset[0,1]$ is a $2^T$-adic interval, then 
\[
\#\{l\in I_2(p)\colon \operatorname{slope}(l)\in J\}\lessapprox (|J|/\delta)^\alpha.
\]
Now, let $v\in S^1$ and let $\delta\leq r\leq 1$. Then there is an index $j$ so that $2^{-jT}\leq r < 2^{-(j-1)T}\leq \delta^{-2\eps}2^{-jT}$, and thus the set of unit vectors $\{v'\in S^1\colon \angle(v',l)\leq r\}$ correspond to a set of slopes $m$ that can be covered by a union of $\approx 1$ intervals of length $2^{-jT}$. We conclude that
\begin{equation}\label{nonConcentration}
\#\{l\in I_2(p)\colon \angle(l,v)\leq r\}\lessapprox (r/\delta)^\alpha.
\end{equation}
Our incidence arrangement $(P_2,L_1,I_2)$ now satisfies Properties (\ref{itemB}), (\ref{itemC}), (\ref{itemD}), and (\ref{nonConcentrationLinesThroughAPointItem}) from Proposition \ref{discretizedFurstenbergWithNonConcentrationProp} (with $\lessapprox$ in place of $\leq \delta^{-\eps}$ in some instances), as well as the estimate $\#I_3\approx \delta^{-3\alpha}$.

\paragraph{Robust transversality of lines through a point.}
Our goal in this section is to find a set of incidences $I_3\subset I_2$ with $\#I_3\approx\delta^{-3\alpha}$ that continues to satisfy Hypotheses (\ref{itemB}), (\ref{itemC}), (\ref{itemD}), and (\ref{nonConcentrationLinesThroughAPointItem}), but which also satisfies a re-scaled version of Property (\ref{robustTransversalityOfLinesThruPtItem}). 

Let $C_1\lesssim 1$ be a constant that will be chosen below. Let $j_0\geq 0$ be an index so that $M_{j_0}>C_1$, and $M_j\leq C_1$ for each $j=0,\ldots,j_0-1$; we can assume that such an index must exist, and furthermore that $j_0\leq k(1-\alpha/2)$, since otherwise we would have 
\[
\#I_2(p)=\#A_p' = \prod_{j=0}^{k-1}M_j =  \prod_{j=0}^{k - \lceil k\alpha/2\rceil}M_j \ \prod_{k + 1 - \lceil k\alpha/2\rceil}^{k-1}M_j \lessapprox C_1^{k}  2^{Tk(\alpha/2)} \lessapprox \delta^{-\alpha/2}\quad\textrm{for all}\ p\in P_2.
\]
But then $\#P_2 \gtrapprox \#I_2 / \delta^{-\alpha/2}\gtrapprox\delta^{-(5/2)\alpha}>\delta^{-2\alpha-c(\alpha)}$, which would establish \eqref{necessaryBdP1}. Define $s = 2^{-j_0T}$. In particular, we have
\begin{equation}\label{lowerBdOnS}
s\geq\delta^{1-\alpha/2}.
\end{equation}

For each $p\in P_2$, define $A_p''=A_p'\cap J$, where $J$ is a $2^T$-adic interval of length $2^{-j_0T}$ that intersects $A_p'.$ By construction we have 
\[
\#A_p''= (M_{0}M_1\cdots M_{j_0-1})^{-1}\#A_p'  \gtrsim \#A_p',
\]
and for any $2^T$-adic interval $J'\subset J$ of length $2^{-(j_0+1)T}$, we have
\[
\#(A_p''\cap J')\leq \frac{1}{C_1}\#A_p''.
\]
Let $I_3\subset I_2$ consist of those pairs $(p,l)\in I_2$ for which $\operatorname{slope}(l)\in A_p''$. We record the following properties of this incidence arrangement.
\begin{itemize}
\item Properties (\ref{itemB}) (with $V=1$) and (\ref{itemC}) continue to hold for $(P_2,L_1,I_3)$, and $\#I_3\gtrapprox\delta^{-3}$. 
\item \eqref{nonConcentration} (and thus Property (\ref{nonConcentrationLinesThroughAPointItem})) continues to hold with $I_3$ in place of $I_2$
\item For each $p\in P_2$, $\#I_3(p) = (\#I_3)/(\#P_2),$ and thus Property (\ref{itemD}) continues to hold.
\item For each $p\in P_2,$ the set $\{ \operatorname{slope}(l)\colon l\in I_3(p)\}\subset [0,1]$ is contained in an interval of length $s$. If $J\subset [0,1]$ is an interval of length $\delta^{2\eps}s$, then 
\begin{equation}\label{robustTransVec}
\#\{ l\in I_3(p)\colon \operatorname{slope}(l)\in J\} \leq \frac{2}{C_1}\#I_3(p).
\end{equation}
(Indeed, the slopes of the lines in the above set correspond to points in $A_p''$ that are contained in a union of at most two $2^T$-adic intervals of length $2^{-(j_0+1)T}$).
\end{itemize}

For each $p\in P_2$, let $v_p\in S^1$ be a vector of the form $(1,m)/\sqrt{1+m^2}$, where $m$ is the midpoint of the smallest dyadic interval containing the set of slopes $\{ \operatorname{slope}(l)\colon l\in I_3(p)\}$. We can suppose $v_p=(\cos\theta,\sin\theta)$, with $\theta\in[0,\pi/4]$.
%%%%%%%%%%

\paragraph{Partitioning into rectangles.}

Let $\mathcal{Z}_0=([0,1]\times[-1,1])\cap (\frac{s}{2}\ZZ)^2$. For each $z=(m,b)\in\mathcal{Z}_0$, let $l_z$ be the line $y=mx+b$, and define 

 \[
P_2^z = \{p\in P_2\colon p\in N_s(l_z),\ \angle(v_p, l_z)\leq s/2\},\quad L_1^z = \bigcup_{p\in P_2^z}I_3(p).%\{l\in L_1\colon \operatorname{dist}(\iota(l), z)\leq 4s\}.
 \]
We claim that
\begin{equation}\label{P3L3Union}
P_2 = \bigcup_{z\in\mathcal{Z}_0}P_2^z,
\end{equation}
and each $p\in P_2$ is contained in at most 9 sets in the above union. Indeed, for each $p\in P_2$, there exists at least one, and at most three $m\in [0,1]\cap (\frac{s}{2}\ZZ)$ with the following property: if $v_m$ is the direction of the vector $(1,m)$, then $\angle (v_p, v_m)\leq s/2$. Once this choice of $m$ has been fixed, there is at least one, and at most three $b\in [-1, 1]\cap(\frac{s}{2}\ZZ)$  with the following property: if $z = (m, b)$, then $p\in N_s(l_z)$. 

Observe as well that each $l\in L_1$ can be contained in at most 100 sets of the form $L_1^z$: if $l\in L_1^z$ then $l\in I_3(p)$ for some $p\in P_2^z$, which implies $\angle(l,l_z)\leq\angle(l, v_p)+\angle(v_p,l_z)\leq s$. Next, since $p\in N_{\delta}(l)\cap N_s(l_z)\cap[0,1]^2$, we have that $\operatorname{dist}(l\cap [0,1]^2,l_z)\leq s+\delta$. But this implies that $\operatorname{dist}(\iota(l),\iota(l_z))\leq 3s$. The result now follows from the fact that $\mathcal{Z}_0$ is $s/2$-separated.

Thus if we define $I_3^z = I \cap (P_2^z\times L_1^z)$, then $I_3 = \bigcup_{z\in\mathcal{Z}_0}I_3^z$, and each $(p,l)\in I_3$ is contained in at most 9 sets in the above union. By refining $\mathcal{Z}_0$ by a factor of at most 100, we can suppose that the sets $\{P_2^z\}$, $\{L_1^z\}$, and thus $\{\mathcal I_3^z\}$ are disjoint. Using dyadic pigeonholing, select $\mathcal{Z}\subset\mathcal{Z}_0$ so that the sets $\{P_2^z\}_{z\in\mathcal{Z}}$, $\{L_1^z\}_{z\in\mathcal{Z}}$ (and thus $\{\mathcal I_3^z\}_{z\in \mathcal{Z}}$) are disjoint, 
$
\sum_{z\in\mathcal{Z}} \#I_3^z\gtrsim \#I_3,
$
and $\#P_2^z$ (and hence $\#I_3^z$) is approximately the same (up to a multiplicative factor of 2) for each $z\in\mathcal{Z}$. In particular,
\begin{equation}\label{eachZisAvg}
\#P_2^z \approx (\#P_2)/(\#\mathcal{Z}),\quad\textrm{and}\quad \# I_3^z\approx (\#I_3)/(\#\mathcal{Z})\quad\textrm{for each}\ z\in\mathcal{Z}. 
\end{equation}

Geometrically, we have partitioned (a large portion of) the arrangement $(P_2, L_1, I_3)$ into sub-arrangements of the form $(P_2^z, L_1^z, I_3^z)$, each of which is contained in the region $[0,1]^2\cap N_{3s}(l_z)$, and each of which contributes a roughly equal number of incidences. Our next step is to analyze each of these sub-arrangements. 

%%%%%%%%%
%%%%%%%%%
%%%%%%%%%

\paragraph{Averaging over segments.}
For the remainder of our arguments we will fix a choice of $z\in\mathcal{Z}$ with 
\begin{equation}\label{lowerBdL1z}
\#L_1^z \gtrapprox (\#L_1)/(\# \mathcal{Z})\gtrapprox \delta^{-2\alpha}(\# \mathcal{Z})^{-1}.
\end{equation}

 Let $P=P_2^z,$  $L=L_1^z$ and $I=I_3^z$. Let $\mathcal{S}$ be the set of $2^T$-adic squares in $[0,1)^2$ of side-length $2^{ (j_0-2-k)T}$ (note that $2^{ (j_0-2-k)T}\approx\delta/s$).  We claim that the arrangement $(P,L,I)$ has the following properties: If $S\in\mathcal{S}$ and if $p\in P\cap S$, then 
 \begin{itemize}
 \item[(i):] There is exactly one line $l\in I(p)$ that contributes points to the square $S$.
 \item[(ii):] Every other line $l'\in I(p)$ is dominated by $l$ on some square $S'\supset S$.
\end{itemize}
Item (i) holds in general (see the discussion following \eqref{necessaryBdP1}). For Item (ii), since $l'\in I(p)$ and $l'\neq l$, we have that $l$ contributed points to the square $S$ containing $p$, and $l'$ did not contribute points to $S$. Let $S'$ be the largest $2^T$-adic square containing $p$ so that $l'$ was dominated by $l$ on $S'$ (the square $S'$ with this property is necessarily unique). We will show that $S'\supset S$. If $S'\supsetneq S$ then we are done. If not, then at the step in the construction of $(P_1,L_1,I_1)$ when $l'$ was processed, the square $S$ would not be contained in a (larger) dominated square. But we also have that 
\[
\operatorname{dist}(l \cap S, l'\cap S)\leq \operatorname{dist}(l \cap S', l'\cap S')\leq (\sqrt 2)2^{-kT},
\] 
and 
\[
\angle(l,l')\leq \angle(l,l_z)+\angle(l_z,l')\leq 2s = 2\cdot 2^{-j_0T} \leq (\frac{\pi}{2})2^{ ([k -(j_0-2)  ] - k) T},
\] 
which implies that $l$ is dominated by $l'$ on square $S$.

Next, observe that if $l,l'$ both contribute points to $S$, then since $\angle(l,l')\leq (\frac{\pi}{2})2^{((j_0+1)-k)T}$, we must have $\operatorname{dist}(l \cap S, l'\cap S)>(\sqrt 2)2^{-kT}.$ 

Let $L(S)\subset L$ be the set of lines that contribute points to the square $S$. Then the above discussion shows that if $l,l'\in L(S)$ are distinct, then the sets $S\cap I(l)$ and $S\cap I(l')$ are disjoint. We have
\[
P = \bigsqcup_{S\in\mathcal{S}} \bigsqcup_{l\in L(S)} (S\cap I(l)).
\]
Since each set $P_l$ is a $(\delta,\alpha)_1$-set, each of the sets $(S\cap I(l))$ have cardinality at most
\[
\delta^{-\eps}\Big(\frac{2^{ (j_0-2-k)T}}{\delta}\Big)^\alpha \lessapprox \Big(\frac{\delta/s}{\delta}\Big)^\alpha = s^{-\alpha}.
\]
Thus by dyadic pigeonholing, there is a number $M$ with
\begin{equation}\label{boundsOnM}
1\leq M\lessapprox s^{-\alpha},
\end{equation}
so that
\[
\#P \approx  \sum_{S\in\mathcal{S}} \sum_{\substack{l\in L(S)\\ M\leq \#(S\cap I(l))<2M}} \#(S\cap I(l)).
\]

Choose $Q\subset P$ by selecting one point from $S\cap I(l)$ for each square $S\in\mathcal{S}$ and each $l\in L(S)$ in the above sum. Define $I_Q = I\cap (Q\times L)$. In particular, \eqref{nonConcentration} and \eqref{robustTransVec} continue to hold with $Q$ in place of $P$, and $I_Q$ in place of $I_2$ or $I_3$. Since for each square $S\in\mathcal{S}$ and each line $l\in L(S)$, we have $\#(S\cap I(l))<2M$, we have retained a factor of $\gtrapprox M^{-1}$ of the incidences. We therefore have 
\begin{equation}\label{lowerBdIQ}
\#I_Q\gtrapprox M^{-1}\delta^{-\alpha}(\#L).
\end{equation}

\paragraph{Rescaling.}
Define $\tilde\delta=\delta/s$. By \eqref{lowerBdOnS} we have $\tilde\delta\leq \delta^{\alpha/2}$, and hence when we write $X\lessapprox Y$, we do not need to distinguish between estimates of the form $\delta^{-C\eps}$ and $\tilde\delta^{-C\eps}$ (recall that in our definition of $\lessapprox$, the constant $C$ is allowed to depend on $\alpha$).

Let $T\colon\RR^2\to\RR^2$ be a translation in the direction $\operatorname{dir}(l_z)^\perp$ that sends $l_z$ to a line through the origin, composed with an anisotropic dilation by $s^{-1}$ in the direction $\operatorname{dir}(l_z)^\perp$; the image of $N_{s}(l)\cap [0,1]^2$ under this linear map is contained in $B(0, 10)$, and it has area $\sim 1$; the map $\iota\circ T\circ\iota^{-1}$ is the composition of a translation with a map that is comparable isotropic dilation by a factor of $s^{-1}$. Let $\tilde Q$ (resp.~$\tilde L$) be the image of $Q$ (resp. $L$) under $T$. Let $\tilde I_Q\subset \tilde Q\times\tilde L$ be the incidence relation corresponding to $I_Q$. If $(\tilde q,\tilde l)\in\tilde I_Q$ then $\tilde q\in N_{\tilde\delta}(\tilde l)$. 

If $S\in\mathcal{S}$ and $l,l'\in L(S)$, then $T(l\cap S)$ and $T(l'\cap S)$ are $\gtrapprox \delta/s=\tilde\delta$ separated. Thus for each such square $S$, $\tilde Q \cap T(S) = T(Q \cap S)$ is $\gtrapprox\tilde\delta$-separated. But since the squares in $\mathcal{S}$ have side-length $\approx\tilde\delta$, each ball $B\subset\RR^2$ of radius $\tilde\delta$ intersect at most $\approx 1$ sets of the form $\{T(S)\colon S\in\mathcal{S}\}$. We conclude that $\#(\tilde Q\cap B)\lessapprox 1$ for each ball $B\subset\RR^2$ of radius $\tilde\delta$. Refining the set $\tilde Q$ by a $\approx 1$ factor, we can suppose that $\tilde Q$ is $\tilde\delta$-separated; this refinement does not affect the validity of any of our previous estimates.

Define $V=\min(1, (s^\alpha M)^{-1})$; by \eqref{boundsOnM} we have $V\approx (s^\alpha M)^{-1}$, and thus $M^{-1}\delta^{-\alpha} \approx V\tilde\delta^{-\alpha}$. We will record the following properties of $\tilde Q,\ \tilde L$, and $\tilde I$. 
\begin{align}
&\#P\gtrsim M\#\tilde Q,\label{controlOfPPrime}\\
&\#\tilde I_Q(q)  = (\#\tilde I)(\#\tilde Q)^{-1}\quad\textrm{for each}\ q\in \tilde Q,\label{hypD}\\
&\#\tilde I_Q(l)\lessapprox M^{-1}\delta^{-\alpha}\approx \tilde\delta^{-\alpha}V\quad\textrm{for each}\ l\in \tilde L,\label{upperBdIncidencesOnLine}\\
&\#\tilde I_Q \gtrapprox M^{-1}\delta^{-\alpha}(\#\tilde L)\gtrapprox \tilde\delta^{-\alpha}V(\#\tilde L).\label{upperBdAvgIncidencesOnLine}
\end{align}
Here, the inequality \eqref{controlOfPPrime} follows from the definition of $Q$, and \eqref{lowerBdIQ} follows from \eqref{lowerBdIQ}.
\eqref{hypD} shows that $(\tilde Q, \tilde L, \tilde I)$ satisfies Property (\ref{itemD}) from Proposition \ref{discretizedFurstenbergWithNonConcentrationProp}. 

For each $\tilde l\in \tilde L$ and each ball $B(x,r)$, with $r\geq\tilde\delta$, 
\begin{equation}\label{pointsOnLineBd}
\#\big(\tilde I(\tilde l) \cap B(x,r)\big)\leq M^{-1}\delta^{-\eps}(r/\delta)^\alpha\approx (r/\tilde\delta)^\alpha V,
\end{equation}
and hence $(\tilde Q, \tilde L, \tilde I)$ satisfies Property (\ref{itemB}) from Proposition \ref{discretizedFurstenbergWithNonConcentrationProp}.

For each ball $B(x,t)$, the set of lines 
\[
\{l\in L\colon \iota(T(l))\in B(x,t)\}=\{T^{-1}(l)\colon \iota(l)\in B(x,t) \}
\] 
is contained in a ball $B'$ of radius $\approx st$; this is because $\iota\circ T\circ\iota^{-1}$ is comparable to a translation composed with an isotropic dilation by a factor of $s^{-1}$. Thus
\begin{equation}\label{pointsInBallBound}
\# (\iota(\tilde L) \cap B(x,t))\leq \delta^{-\eps}(st/\delta)^{2\alpha}=\delta^{-\eps}(t/\tilde\delta)^{2\alpha},
\end{equation}
and hence $(\tilde Q, \tilde L, \tilde I)$ satisfies Property (\ref{itemC}) from Proposition \ref{discretizedFurstenbergWithNonConcentrationProp}

Next, for each $q\in\mathcal{Q}$ we have the following analogue of \eqref{nonConcentration}: for each $v\in S^1$ and each $r\geq\tilde\delta$, 
\begin{equation}\label{rescaledNonConcentration}
\#\{\tilde l \in \tilde I_Q(\tilde q)\colon  \angle(\tilde l,v)\leq r\}\lessapprox (sr/\delta)^\alpha\leq (r/\tilde\delta)^\alpha.
\end{equation}
Indeed, \eqref{rescaledNonConcentration} follows from \eqref{nonConcentration}, since if $q=T^{-1}(\tilde q),$ then $\tilde I_Q(\tilde q)=\{T(l)\colon l\in I_Q(q)\}$. Since the latter set of lines obeys \eqref{nonConcentration}, the former obeys \eqref{rescaledNonConcentration}. Thus $(\tilde Q, \tilde L, \tilde I)$ satisfies Property (\ref{nonConcentrationLinesThroughAPointItem}) from Proposition \ref{discretizedFurstenbergWithNonConcentrationProp}.

Finally, if the constant $C_1$ is chosen appropriately, then for each $\tilde q\in\tilde Q$ we have the following analogue of \eqref{robustTransVec}: If $v\in S^1$, then since $\tilde\delta\leq\delta^{\alpha/2}$, we have 
\begin{equation}\label{rescaledRobustTransVec}
\#\{\tilde l\in \tilde I_Q(\tilde q)\colon \angle(\tilde l, v)\leq \tilde \delta^{4\eps/\alpha}\}  \leq \frac{1}{10}\#I_Q(\tilde q).
\end{equation}
Again, \eqref{rescaledRobustTransVec} follows from \eqref{robustTransVec} by the same reasoning as above. Thus $(\tilde Q, \tilde L, \tilde I)$ satisfies Property (\ref{robustTransversalityOfLinesThruPtItem}) from Proposition \ref{discretizedFurstenbergWithNonConcentrationProp} (with $4\eps/\alpha$ in place of $\eps$, but this is harmless since $\alpha$ is fixed). 
\begin{remark}\label{surviveRefinement}
Observe that arrangement $(\tilde Q,\tilde L, \tilde I)$ satisfies Property (\ref{robustTransversalityOfLinesThruPtItem}) with room to spare---if we select a smaller set of incidences $\tilde I_Q'\subset\tilde I_Q$ so that
\[
\#\tilde I'_Q(\tilde q) \geq \frac{1}{4} \#\tilde I_Q(\tilde q) \quad\textrm{for each}\ q\in\tilde Q,
\]
then Property (\ref{robustTransversalityOfLinesThruPtItem}) will remain true with $\tilde I_Q'$ in place of $\tilde I_Q$; we will need this flexibility below. 
\end{remark}

We have shown that the arrangement $(\tilde Q, \tilde L, \tilde I_Q)$ satisfies all of the hypotheses of Proposition \ref{discretizedFurstenbergWithNonConcentrationProp} (or more precisely, an equivalent version of these hypotheses, as discussed in Remark \ref{epsVsLessapprox}) except Property (\ref{itemA}), which is only satisfied on average. To fix this, we shall use the following graph refinement lemma from \cite{dg}.

\begin{lemma}[Graph refinement]\label{graphRefinementLemma}
Let $G = (A\sqcup B, E)$ be a bipartite graph. Then there is a sub-graph $G'=(A'\sqcup B', E')$ so that $\#E'\geq \#E/2$; each vertex in $A'$ has degree at least $\frac{\#E}{4\#A}$; and each vertex in $B'$ has degree at least $\frac{\#E}{4\#B}$.
\end{lemma}

Apply Lemma \ref{graphRefinementLemma} to $(\tilde Q\sqcup \tilde L,\ \tilde I),$ and let $(\dbtilde Q\sqcup \dbtilde L,\ \dbtilde I)$ be the resulting refinement.  We can verify that $(\dbtilde Q\sqcup \dbtilde L,\ \dbtilde I)$ now satisfies Properties (\ref{itemA}) - (\ref{nonConcentrationLinesThroughAPointItem}) of Proposition \ref{discretizedFurstenbergWithNonConcentrationProp}. Indeed, by \eqref{hypD}, each $\dbtilde q\in\dbtilde Q$ satisfies 
\[
\frac{1}{4}(\#\tilde I)(\#\tilde Q)^{-1}\leq  \# \dbtilde{I}(\dbtilde q) \leq (\#\tilde I)(\#\tilde Q)^{-1},
\]
and hence  Property (\ref{robustTransversalityOfLinesThruPtItem}) remains true (see the discussion in Remark \ref{surviveRefinement}). \eqref{rescaledNonConcentration}, and thus Property (\ref{nonConcentrationLinesThroughAPointItem}) remains true for similar reasons. Similarly for Property (\ref{itemD}). Properties (\ref{itemB}) and (\ref{itemC}) remain true because  $(\dbtilde Q\sqcup \dbtilde L,\ \dbtilde I)$ is a refinement of $(\tilde Q\sqcup \tilde L,\ \tilde I)$. 

Next, we will verify that Property (\ref{itemA}) holds. Indeed, for each $\dbtilde l\in \dbtilde L$ we have
\[
\tilde\delta^{-\alpha}V \lessapprox \#\tilde I(\#\tilde L)^{-1}\lesssim \#\dbtilde I(\dbtilde l)\leq \#\tilde I(\dbtilde l) \lessapprox \tilde\delta^{-\alpha}V.
\]
The first inequality follows from \eqref{upperBdAvgIncidencesOnLine}. The second inequality follows from Lemma \ref{graphRefinementLemma}. The final inequality follows from \eqref{upperBdIncidencesOnLine}. 

Note as well that \eqref{lowerBdL1z}, \eqref{upperBdIncidencesOnLine} and \eqref{upperBdAvgIncidencesOnLine} imply that 
\begin{equation}\label{lowerBdSizeTildeLP}
\#\dbtilde L\gtrapprox \delta^{-2\alpha}(\#\mathcal{Z})^{-1}
\end{equation}
(in brief, \eqref{lowerBdL1z} says that $\#\tilde L\gtrapprox \delta^{-2\alpha}(\#\mathcal{Z})^{-1},$ while \eqref{upperBdIncidencesOnLine} and \eqref{upperBdAvgIncidencesOnLine} says that each line in $\#\tilde L$ contributes an approximately equal number of incidence. Thus the cardinality of $\tilde L$ cannot decrease substantially after application of the graph refinement lemma).

Applying Proposition \ref{discretizedFurstenbergWithNonConcentrationProp}, we conclude that
\begin{equation}
\# \tilde Q \geq \#\dbtilde Q \gtrapprox \tilde\delta^{-2\alpha-c(\alpha)}(\tilde\delta^{2\alpha}\#\dbtilde L)V^{c(\alpha)/\alpha}
\gtrapprox \tilde\delta^{-2\alpha-c(\alpha)}(\tilde\delta^{2\alpha}\delta^{-2\alpha}(\#\mathcal{Z})^{-1})V^{c(\alpha)/\alpha},
\end{equation}
where the final inequality used \eqref{lowerBdSizeTildeLP}. Thus by \eqref{eachZisAvg} and \eqref{controlOfPPrime} (and using the fact that $V\geq V^{c(\alpha)/\alpha}$ since $V\geq 1$ and $c(\alpha)/\alpha\leq1)$, we have
\begin{equation*}
\begin{split}
\#P_2 &\gtrapprox (\#\mathcal{Z})(\# P)\\
& \gtrapprox (\#\mathcal{Z}) M \tilde\delta^{-2\alpha-c(\alpha)}(\tilde\delta^{2\alpha}\delta^{-2\alpha}(\#\mathcal{Z})^{-1})V\\
& \gtrapprox M (\delta/s)^{-2\alpha-c(\alpha)}(s^{-2\alpha})V^{c(\alpha)/\alpha}\\
& \gtrapprox M \delta^{-2\alpha-c(\alpha)} s^{c(\alpha)}(s^\alpha M)^{-c(\alpha)/\alpha}\\
& \gtrapprox \delta^{-2\alpha-c(\alpha)} M^{1-c(\alpha)/\alpha}\\
& \geq \delta^{-2\alpha-c(\alpha)}.
\end{split}
\end{equation*}
Since $P_2\subset P_1$, we obtain \eqref{necessaryBdP1}.
\end{proof}

\appendix

\section{Proof of Proposition~\ref{modifiedGKZ}}\label{modifiedGKZAppendix}
In this section, we will sketch the proof of Proposition~\ref{modifiedGKZ}. As discussed in Section \ref{toolsFromAddComb}, this proof does not contain any new ideas; it is simply the proof of Theorem 1.1 from \cite{gkz}, except we keep track of the dependence on $K_0,\ldots,K_4$. We also preserve two distinct estimates that were merged too early in \cite{gkz}, which allows for a slightly stronger final bound. The proof sketch presented here is not self contained---the reader should compare with Theorem 1.1 from \cite{gkz}. 

To prove Proposition~\ref{modifiedGKZ}, it suffices to establish the estimate
\begin{equation}\label{modifiedPropmodifiedGKZ}
K_1^6 K_2^{6+\alpha} K_3^{2(9+4\alpha)} K_4^{4(2+\alpha)} \gtrapprox  \delta^{-\alpha(1-\alpha)}.
\end{equation}
To begin, we will record some facts about the set $A$ that were proved in \cite{gkz}. For comparison, the set $A$ in the statement of Proposition~\ref{modifiedGKZ} corresponds to the set $A'$ from \cite{gkz}. (In \cite{gkz}, $A'\subset A$ was selected so that $\mathcal{E}_{\rho}(A'+A')\leq K_3\mathcal{E}_{\rho}(A')$ for all $\delta\leq\rho\leq 1$; in Proposition~\ref{modifiedGKZ}, this property already holds for $A$.) 

The authors in \cite{gkz} use their assumed upper bound on $\mathcal{E}_{\delta}(A)$ to obtain the following lower bound on the multiplicative energy of $A$:
\begin{equation}\label{sumaAbA}
\sum_{a,b\in A}\#\big((a.A)_{\delta^+} \cap (b.A)_{\delta^+}\big)\gtrsim K_4^{-1}(\#A)^3.
\end{equation}
We claim that our hypothesis $\mathcal{E}_{\delta}^\times(A,A)\geq K_4^{-1}(\#A)^3$ implies the analogous statement, with $\delta$ replaced by $2\delta$. Indeed, for each $a,b\in A$, select $a',b'\in A_{\delta^\times}$ with $|a-a'|< 2\delta,\ |b-b'|<2\delta$. Then 
\[
\#\big((a.A)_{2\delta^+} \cap (b.A)_{2\delta^+}\big)\gtrsim \#\big( (a'. A_{\delta^\times})\cap (b'. A_{\delta^\times}).
\]
By \eqref{sumaAbA} (with $2\delta$ in place of $\delta$) and pigeonholing, we conclude that there exists $b\in A$ so that
\begin{equation}\label{analogueOf8}
\sum_{a\in A}\#\big((a.A)_{2\delta^+} \cap (b.A)_{2\delta^+}\big)\gtrsim K_4^{-1}(\#A)^2.
\end{equation}
\eqref{analogueOf8} is the analogue of (8) from \cite{gkz}. From here, the argument proceeds in essentially an identical fashion, except we track the dependence on $K_1,K_2,K_3,K_4$. We will briefly highlight a few of the key estimates from \cite{gkz}, and give the analogous statements. 

The authors in \cite{gkz} select a number $K_4^{-1}\leq\rho\leq 1$ and a set $\bar A\subset A$, with $\#\bar A\gtrsim|\log\delta|^{-1}(\#A)$, so that 
\[
\#\big( (aA)_{2\delta^+}\cap (bA)_{2\delta^+}\big)\sim (\# A)K_4^{-1}\rho^{-1}\quad\textrm{for each}\ a\in \bar A. 
\]
They then select a set $A_1\subset\bar A$ with $\#A_1\gtrapprox\#\bar A$, so that $A_1+A_1$ has small $\rho$-covering number for all $\delta\leq\rho \leq 1$. The key estimates we need are the following analogues of Lemmas 3.6 and 3.7 from \cite{gkz}. Let $a_1,a_2,b_2,b_2\in A_1$ and let $d_1=a_1-b_1,$ $d_2 = a_2-b_2$. Then
\begin{equation}\label{analogue12}
\mathcal{E}_{\delta}(d_1A+ d_2A)\lessapprox K_1K_2K_3^{8}K_4^4\rho^{4} \max(|d_1|,|d_2|)^{\alpha}\#(A),
\end{equation}
and for each $k\geq 2$ there is a set $A_2\subset A_1$, with $\#A_2\geq\frac{1}{4}\#A_1$, so that
\begin{equation}\label{analogue14}
|d_1A_2+ \underbrace{d_2A_2 +\ldots + d_2A_2}_{\text{$k$ times}}|\lessapprox K_1K_2K_3^{7+k}K_4^4\rho^{5-k} \max(|d_1|,|d_2|)^{\alpha}\#(A). 
\end{equation}

Next, let $s$ and $\gamma$ be parameters that we will specify below. Proceeding as in \cite{gkz}, we define
\[
B = \Big\{\frac{a_1-a_2}{a_3-a_4}\colon a_i\in A_1,\ |a_3-a_4|>\delta^\gamma\Big\}.
\]
Lemma 4.1 from \cite{gkz} says that either there is a point $b\in B\cap[0,1]$ so that both $b/2$ and $(b+1)/2$ have distance at least $s$ from $B$, or the $s$-covering number of $B\cap[0,1]$ is comparable to $s^{-1}$. The former situation is called the ``gap case,'' while the latter is called the ``dense case.''

Our arguments in both of these cases mirror those of \cite{gkz}, with \eqref{analogue12} and \eqref{analogue14} in place of Lemmas 3.6 and 3.7. The arguments from \cite{gkz} show that in the dense case, we must have either

\begin{equation}\label{densePossibility}
K_2K_3^8K_4^4 \gtrapprox \delta^{\alpha}s^{-1}\quad\textrm{and/or}\quad K_1^3K_2K_3^8K_4^4 \gtrapprox \delta^{-\gamma\alpha},
\end{equation}
while in the gap case we must have 
\begin{equation}\label{gapPossibility}
K_1^3K_2K_3^{10}K_4^4 \gtrapprox \left(\delta^{\gamma-1}s\right)^{\alpha}.
\end{equation}
Until this point, the arguments above have proceeded in parallel to those in \cite{gkz}. The improvement over \cite{gkz} comes from properly choosing $s$ and $\gamma$; we choose $s$ and $\gamma$ so that
\begin{equation}\label{defnSAndGamma}
s = \left(K_1^3K_2^2K_3 \delta^{3\alpha/2}\right)^{2/(\alpha+2)},\quad \delta^{\gamma} = K_3^{1/\alpha}\delta^{1/2}s^{-1/2}.
\end{equation}
We can suppose that $0<s<1$, since otherwise $K_1^3K_2^2K_3>\delta^{-3\alpha/2}$, which is a stronger estimate than \eqref{modifiedPropmodifiedGKZ} (recall that by assumption, $K_4\geq 1$). Similarly, we can suppose that $0<\delta^\gamma<1$ (and hence $\gamma>0$), since otherwise $K_3^{2/\alpha}\delta^{1-\alpha} \geq 1,$ and hence $K_3\geq \delta^{-\frac{\alpha(1-\alpha)}{2}},$ which again is a stronger estimate than \eqref{modifiedPropmodifiedGKZ} (recall that by assumption, $K_1,K_2,K_4\geq 1$ and $\alpha\in(0,1)$).

Since at least one of the three inequalities from \eqref{densePossibility} and \eqref{gapPossibility} must hold, we can verify that with this choice of $s$ and $\gamma$, \eqref{densePossibility} and \eqref{gapPossibility} imply  \eqref{modifiedPropmodifiedGKZ}.

\end{document}